\documentclass{amsart}

\RequirePackage{color}

\def\smallskip{\vskip\smallskipamount}
\def\medskip{\vskip\medskipamount}
\def\bigskip{\vskip\bigskipamount}

\usepackage{amsmath,amsthm,bm,mathrsfs,amscd,tikz}
\usepackage{ytableau}
\usepackage{comment}
\usepackage{mathdots}
\usepackage[all]{xy}
\usepackage{graphicx}
\usepackage[colorinlistoftodos]{todonotes}
\usepackage{enumerate}
\usepackage{feynmp-auto}
\usepackage[english]{babel}
\usepackage[utf8]{inputenc}
\usepackage{float}
\usepackage{tikz}
\usepackage{tikz-cd}
\usepackage{amssymb}
\usepackage{mathtools}
\usetikzlibrary{decorations.pathmorphing}

\usepackage[utf8]{inputenc}

\newtheoremstyle{thmstyle}{}{}{\itshape}{}{\bfseries}{ }{5pt}{}
\newtheoremstyle{exstyle}{}{}{}{}{\bfseries}{ }{5pt}{}
\newtheoremstyle{defstyle}{}{}{}{}{\bfseries}{ }{5pt}{}
\newtheoremstyle{remstyle}{}{}{}{}{\bfseries}{ }{5pt}{}

\theoremstyle{thmstyle}
\newtheorem{thm}{Theorem}[section]
\newtheorem{theorem}[thm]{Theorem}

\newtheorem{lemma}[thm]{Lemma}

\newtheorem{proposition}[thm]{Proposition}

\newtheorem{cor}{Corollary}[section]
\newtheorem{corollary}[thm]{Corollary}

\theoremstyle{exstyle}

\newtheorem{example}[thm]{Example}

\theoremstyle{defstyle}

\newtheorem{definition}[thm]{Definition}

\newtheorem{def-prop}[thm]{Definition-Proposition}
\newtheorem{def-lem}[thm]{Definition-Lemma}
\newtheorem{rem-convention}[thm]{Remark-Convention}
\newtheorem{def-note}[thm]{Definition-Notation}

\theoremstyle{remstyle}

\newtheorem{remark}[thm]{Remark}

\theoremstyle{remstyle}

\newcommand{\Ext}{\operatorname{Ext}}

\newcommand{\Z}{\mathbb{Z}}
\newcommand{\R}{\mathbb{R}}
\newcommand{\K}{\mathbb{K}}

\newcommand{\Hom}{\rm Hom}

\DeclareMathOperator*{\Ker}{Ker}

\DeclareMathOperator*{\rep}{rep}

\makeatletter
\newcommand{\doublewidetilde}[1]{{%
  \mathpalette\double@widetilde{#1}%
}}
\newcommand{\double@widetilde}[2]{%
  \sbox\z@{$\m@th#1\widetilde{#2}$}%
  \ht\z@=.9\ht\z@
  \widetilde{\box\z@}%
}
\makeatother

\begin{document}

\title[Free products of semisimple algebras via quivers]{Representations of free products of semisimple algebras via quivers}
\author[A. Buchanan, I. Dimitrov, O. Grace, C. Paquette, D. Wehlau, T. Xu]{Andrew Buchanan, Ivan Dimitrov, Olivia Grace, Charles Paquette, David Wehlau, Tianyuan Xu} 

\address{Department of Mathematics and Statistics, Queen's 
University, Kingston ON, Canada}
\email{16maib@queensu.ca}
\address{Department of Mathematics and Statistics, Queen's 
University, Kingston ON, Canada}
\email{dimitrov@queensu.ca}
\address{Department of Mathematics and Statistics, Queen's 
University, Kingston ON, Canada}
\email{17ofg@queensu.ca}
\address{Department of Mathematics and Computer Science, Royal Military College of Canada, Kingston ON, Canada}
\email{charles.paquette.math@gmail.com}
\address{Department of Mathematics and Computer Science, Royal Military College of Canada, Kingston ON, Canada}
\email{david.wehlau@gmail.com}
\address{Department of Mathematics, University of Colorado Boulder, Boulder CO 80309, USA}
\email{tianyuan.xu@colorado.edu}



\maketitle

\begin{abstract}
    Let $\K$ denote an algebraically closed field and $A$ a free product of finitely many semisimple associative $\K$-algebras. We associate to $A$ a finite acyclic quiver $\Gamma$ and show that the category of finite dimensional $A$-modules is equivalent to a full subcategory of the category ${\rm rep}(\Gamma)$ of finite dimensional representations of $\Gamma$. Under this equivalence, the simple $A$-modules correspond exactly to the $\theta$-stable representations of $\Gamma$ for some stability parameter $\theta$. This gives us necessary conditions for an $A$-module to be simple, conditions which are also sufficient if the module is in general position. Even though there are indecomposable modules that are not simple, we prove that a module in general position is always semisimple. We also discuss the construction of arbitrary finite dimensional modules using nilpotent representations of quivers. Finally, we apply our results to the case of a free product of finite groups when $\K$ has characteristic zero. 
    
\end{abstract}
\tableofcontents

\section{Introduction}

Let $\K$ be an algebraically closed field, let $m$ be a positive integer, and let $A_1,A_2,..., A_m$ be finite dimensional semisimple (associative, unital) $\K$-algebras. 
In this paper, we are interested in studying the (finite dimensional) representation theory of algebras of the form 
$$A:= A_1 * A_2 * \cdots * A_m,$$
where the symbol $*$ stands for free product. 
Our motivation comes from the special case where 
$A_i$ is the group algebra $\K G_i$ of an arbitrary finite cyclic group $G_i$ for each $i$. More specifically, in 
\cite{DPWX} a subset of the authors studied the representation theory of subregular $J$-rings of Coxeter systems. These are certain subrings of Lusztig's asymptotic Hecke algebras and properly include,
up to Morita equivalence, all algebras of the form $A=\K G_1* \K G_2* \cdots \K G_m$. 

In the above special case, the algebra $A=\K G_1* \K G_2* \cdots \K G_m$ is
naturally isomorphic to the group algebra $\K G$ of the group $G=G_1 * G_1 * \cdots G_k$ of the free product of the $G_i$. Interesting examples of groups of that form include the infinite dihedral group $\mathbb{Z}_2 * \mathbb{Z}_2$ and the projective modular group ${\rm PSL}_2(\mathbb{Z}) \cong \mathbb{Z}_2 * \mathbb{Z}_3$. We note that for the case $m=2$, the representation theory of the group algebra $A=\K G=\K(G_1*G_2)$ (where $G_1$ and $G_2$ are cyclic) has 
 been studied by Adriaenssens--Le Bruyn \cite{AL} and Sletsjoe \cite{Slet}. However, to our knowledge, not much is known about the representation theory of $A$ in our more general setting where the number $m$ can be any positive integer and each $A_i$ can be any semisimple algebra.

\medskip

In the $\K$-linear category of associative unital $\K$-algebras, the free product $A=A_1 * A_2 * \cdots * A_m$ coincides with the coproduct of the family of algebras $\{A_1, \ldots, A_m\}$.  The free product does not depend on the order of the factors, and for an algebra $B$, we have that $B*\K \cong B$. Therefore, an $A_i$ that is one dimensional will be called a \emph{trivial factor}. For the purpose of studying the modules over $A$, we will therefore assume that there is no trivial factor in $A$ (and in particular, $A$ is not one dimensional), unless otherwise stated. By the universal property of the coproduct, if $V$ is a finite dimensional $\K$-vector space, then the data of an algebra homomorphism $A \to {\rm End}_\K(V)$ is equivalent to the data of algebra homomorphisms $A_i \to {\rm End}_\K(V)$ for all $i$. Giving an $A$-module structure on $V$ is therefore equivalent to simultaneously giving an $A_i$-module structure on $V$ for all $i$. If each $A_i$ is semisimple, then for each $i$, we can decompose $V$ as a finite direct sum of simple $A_i$-modules.
If we are given these decompositions for each $i$, and how each direct summand of each decomposition embeds into $V$, then we can recover the $A$-module structure on $V$. We make this observation explicit by showing that there is an equivalence between the category $\rep(A)$ of finite dimensional left $A$-modules and a subcategory $\mathcal{C}$ of representations of a suitable finite acyclic quiver $\Gamma = \Gamma_A$ that we call a \emph{generalized subspace quiver}. We will define a functor $F:\rep(A)\to\mathcal{C}$ in Section \ref{sec:FandG} to establish this equivalence.

\medskip

There is a stability parameter $\theta$ on the quiver $\Gamma$ such that all representations in the subcategory $\mathcal{C}$ are $\theta$-semistable. Moreover, the simple $A$-modules correspond precisely to the $\theta$-stable representations in $\mathcal{C}$ via the functor $F$. We make use of the well-developed theory of $\theta$-stable decompositions for acyclic quivers, due to Derksen and Weyman \cite{DW}, to determine all $\theta$-stable dimension vectors. Under the correspondence mentioned above, consideration of the $\theta$-stable dimension vectors leads to our main theorem on simple $A$-modules:
given a finite dimensional $A$-module $M$ and the decomposition of $M$ as a direct sum of simple $A_i$-modules for each $i$, our main theorem gives a convenient necessary criterion to decide if $M$ is simple as an $A$-module. The criterion is numerical, and it also turns out to be sufficient whenever $M$ is in general position, that is, whenever $M$ lies on some open and non-empty set of some irreducible variety which parameterizes the modules of a given dimension; see Theorem \ref{mainThm}.

\medskip

Using representations of $\Gamma$, we also prove the somewhat surprising result that a left $A$-module in general position is completely reducible. As a consequence, we deduce that the dimension vectors of the simple $A$-modules are exactly the Schur roots of the quiver $\Gamma$ satisfying $m$ linear equalities; see Corollary \ref{coro:schur}. 

\medskip

We note that although modules in general positions are semisimple, there always exist indecomposable modules that are not simple, when the free product has more than one (non-trivial) factors. 
With the only exception of the group algebra of the infinite dihedral group, all categories $\rep(A)$ are of (strictly) wild representation type. Moreover, nilpotent representations of quivers can be used to construct all indecomposable finite dimensional $A$-modules.

\medskip

Then, we will use moduli spaces of quivers to compute the number of parameters needed to parameterize all of the simple $A$-modules. This yields an explicit formula in the case where every $A_i$ is basic.

\medskip

Finally, we will apply our result in the case of a free product of finite cyclic groups, in which case our numerical criterion for simplicity of $A$-modules can be formulated in terms of the dimensions of the eigenspaces of the $A$-module $M$ seen as an $A_i$-module for $i=1,2,\ldots, m$. We note that in the particular case of a free product of two finite cyclic groups, our theorem (Theorem \ref{MainTheoCyclic}) has also been obtained by Adriaenssens and Le Bruyn in \cite{AL}. 

\medskip

Generalized subspace quivers and similar stability parameters, with their stable representations, have recently been studied in \cite{CJ} to analyse some problems coming from frame theory in analysis. We hope that our results can also lead to some interesting consequences in analysis.

\subsection*{Acknowledgement} The second, fourth and fifth named authors are supported by the
National Sciences and Engineering Research Council of Canada. The fourth and fifth named authors are also supported by the Canadian Defence
Academy Research Programme. We thank Kaveh Mousavand for helpful discussions.

\section{Some quiver representation theory}

We let $m$ be a positive integer and let $A=A_1* A_2* \ldots* A_m$ be a free product of $m$ finite dimensional (associative and unital) semisimple $\K$-algebras $A_1, A_2, \ldots, A_m$. Since $\K$ is algebraically closed, each semisimple algebra $A_i$ is isomorphic to a finite product of matrix algebras over $\K$ by the Wedderburn-Artin Theorem. Hence, we may assume that
$$A_i = A_{i1} \times \cdots \times A_{ir_i}$$
for some positive integer $r_i$ where each $A_{ij}$ is a matrix algebra over $\K$. We let $w_{ij}$ denote the dimension of the unique simple $A_{ij}$-module. We will call $w_{ij}$ a \emph{weight} of $A$ or of $A_i$ and say that it is \emph{trivial} when it is one. Equivalently, we assume that $A_{ij}$ is the matrix algebra ${\rm Mat}_{w_{ij}}(\K)$ of $w_{ij} \times w_{ij}$ matrices over $\K$. For each $i,j$, and $1 \le p,q \le w_{ij}$, we let $e_{ij}^{pq}$ denote the matrix in $A_{ij}$ having zeros everywhere except in row $p$ and column $q$, where the entry is 1. We let $E_{ij}$ denote the identity of $A_{ij}$, so that
  $$E_{ij}=\sum_{k=1}^{w_{ij}}e_{ij}^{kk}.$$ 
  
  Let $S$ be any left $A_{ij}$-module. By basic linear algebra, we have $e_{ij}^{k1}S = e_{ij}^{kk}S$ for all $1 \le k \le w_{ij}$, and there is a $\K$-vector space decomposition
  \begin{equation}
      \label{eq:decomp}
      S=\bigoplus_{k=1}^{w_{ij}}e_{ij}^{kk}S=\bigoplus_{k=1}^{w_{ij}}e_{ij}^{k1}S=\bigoplus_{k=1}^{w_{ij}}e_{ij}^{k1}e_{ij}^{11}S.
  \end{equation}
  It further follows that if 
 $\{b_1, \ldots, b_s\}$ is a basis of $e_{ij}^{11}S$, then the set $\{e_{ij}^{k1}b_l \mid 1 \le k \le w_{ij},\, 1 \le l \le s\}$ forms a basis for $S$. 

\medskip

\subsection{The quiver $\Gamma_A$} Let $A=A_1* \ldots* A_m$ be as before. We associate a quiver 
$\Gamma= \Gamma_A = (\Gamma_0, \Gamma_1)$ to $A$ as follows: the quiver $\Gamma$ has a unique sink vertex denoted $v_0$ and one source vertex $v_{ij}$ for each $i,j$ with $1 \le i \le m$ and $1 \le j \le r_i$. We put $w_{ij}$ arrows from $v_{ij}$ to $v_0$ for all $i,j$. The collection $\{v_{i1}, \cdots, v_{i r_i} \}$ of vertices is called the \emph{$i$-th cluster} of $\Gamma$. The arrows starting at $v_{ij}$ are labelled $\alpha_{ij}^k$ where $1 \le k \le w_{ij}$. We refer to $\Gamma_A$ as a \emph{generalized subspace quiver} of $A$, since when $w_{ij}=1$ for all $i,j$ we recover a classical subspace quiver.
\medskip

Recall that a \emph{representation} $M = (\{M(v)\}_{v \in \Gamma_0},\{M(\alpha)\}_{\alpha \in \Gamma_1})$ of any quiver $\Gamma$ is a collection $\{M(v)\}_{v \in \Gamma_0}$ of $\K$-vector spaces together with a collection $\{M(\alpha)\}_{\alpha \in \Gamma_1}$ of $\K$-linear maps consisting of one linear map $M(\alpha): M(u) \to M(v)$ for each arrow $\alpha: u\to v$ in $\Gamma$. We say $M$ is \emph{finite dimensional} when all the vector spaces $M(v)$ are finite dimensional, and in this case we define the \emph{dimension vector} of $M$ to be the vector $d=(d_v)_{v\in \Gamma_0}$ where $d_v=\dim_\K M(v)$ for each $v\in  \Gamma_0$; we also denote $d$ by $\dim(M)$. We will also refer to any element of $\mathbb{Z}_{\ge 0}^{\Gamma_0}$ as being a dimension vector, since it is the dimension vector of at least one representation.

\medskip

The finite dimensional representations of $\Gamma$ form a category where a {morphism} $f: M \to N$ for two representations of $\Gamma$ is a collection $\{f_v\}_{v \in \Gamma_0}$ of $\K$-linear maps such that $N(\alpha)f_u = f_vM(\alpha)$ for each arrow $\alpha: u \to v$ in $\Gamma$. We denote this category by $\rep(\Gamma)$. It is a $\K$-linear Hom-finite abelian category.

\begin{remark}[Notation] 
\label{rmk:notation}
 To simplify notation, for each representation $M$ of $\Gamma_A$ and its dimension vector $d$, we will often denote $M(v_0), M(v_{ij}), d_{v_0}$ and  $d_{v_{ij}}$ by $M_0, M_{ij}, d_0$ and $d_{ij}$, respectively. Similarly, for each morphism $f=(f_v)_{v\in \Gamma_0}:M\to N$ in $\rep(\Gamma_A)$, we will often denote $f_{v_0}$ by $f_0$ and $f_{v_{ij}}$ by $f_{ij}$ for all $i,j$. 
 \end{remark}

\begin{example}
Consider the algebra $A = (\K^2 \times {\rm Mat}_2(\K))*(\K^3)$.
Here $m=2$, $r_1=r_2=3$, 
$w_{11}=w_{12}=w_{21}=w_{22}=w_{23}=1$ and $w_{13}=2$.
The quiver $\Gamma$ is as follows,
$$
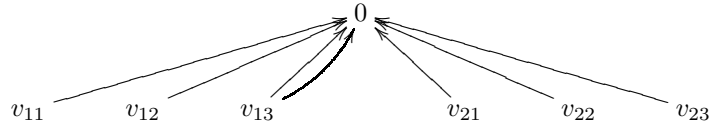
\xymatrix{&&& 0 &&& \\ v_{11} \ar[urrr] & v_{12} \ar[urr] & v_{13} \ar@/_/[ur] \ar[ur] & & v_{21} \ar[ul] & v_{22} \ar[ull] & v_{23} \ar[ulll]}$$
and $\Gamma$ has $m=2$ clusters, namely $\{v_{11},v_{12}, v_{13}\}$ and $\{v_{21},v_{22}, v_{23}\}$. 
\end{example}

\medskip

\subsection{The functor $F$} 
\label{sec:FandG}
Let $\rep(A)$ denote the $\K$-linear category of finite dimensional left $A$-modules (or representations of $A$). We now define a $\K$-linear covariant functor $F: \rep(A) \to \rep(\Gamma)$. 

\medskip

Let $M$ be a finite dimensional $A$-module. To associate a representation $F(M)$ of $\Gamma$ to $M$, we first let $F(M)_0=M$ and $F(M)_{{ij}} =e_{ij}^{11}M$ (see Remark \ref{rmk:notation}). For each arrow $\alpha_{ij}^k:v_{ij}\to v_0$, we define  the map $F(M)(\alpha_{ij}^k):e_{ij}^{11}M\to M$ to be the composition $\iota_{ij}^k\circ \phi_{ij}^k$ where $\phi_{ij}^k$ is the vector space isomorphism $e_{ij}^{11}M \to e_{ij}^{k1}M=e_{ij}^{kk}M$ given by left multiplication by $e_{ij}^{k1}$ and $\iota_{ij}^k$ is the natural embedding of $e_{ij}^{kk}M$ into $M$. We have defined the representation $F(M)$. 

For each homomorphism $f: M \to N$ of $A$-modules, we define a morphism $F(f): F(M)\to F(N)$ in $\rep(\Gamma)$ as follows. We first let $F(f)_0=f$. For each $i,j$, we have $f(e_{ij}^{11}x) = e_{ij}^{11}f(x) \in e_{ij}^{11}N$ for all $x\in M$, therefore $f$ restricts to map from $e_{ij}^{11}M=F(M)_{ij}$ to $e_{ij}^{11}N=F(N)_{ij}$. We define  $F(f)_{ij}$ to be this restriction. Using the fact that $f$ is a module homomorphism, it is straightforward to check that for each $i,j,k$, we have that $F(N)(\alpha_{ij}^k)F(f)_{ij} = F(f)_{0}F(M)(\alpha_{ij}^k)$, therefore $F(f)$ is indeed a morphism in $\rep(\Gamma)$. 

Finally, it is not hard to check that we have defined a covariant $\K$-linear functor $F: \rep(A) \to \rep(\Gamma)$. We note that $F$ maps $\rep(A)$ to a proper subcategory of $\rep(\Gamma)$. In particular, Equation \eqref{eq:decomp} implies that the dimension vector of $F(M)$ is balanced in the following sense for every $A$-module $M$:
\begin{definition}
\label{def:balanced}
We say that a dimension vector $d\in (\Z_{\ge 0})^{\Gamma_0}$ is \emph{balanced} if $\sum_{j=1}^{r_i} w_{ij}d_{{ij}}=d_{0}$ for every $1\le i\le m$.  
\end{definition}

To describe the image of the functor $F$ more precisely, we associate to each representation $X\in \rep(\Gamma)$ and each cluster $i$ in $\Gamma$ the map
\begin{equation}
    \label{eq:Psi}
\Psi(X,i):=\bigoplus_{j=1}^{r_i}\bigoplus_{k=1}^{w_{ij}} X(\alpha_{ij}^k): \left(\bigoplus_{j=1}^{r_i}\bigoplus_{k=1}^{w_{ij}} X_{ij}\right) \longrightarrow X_0.
\end{equation}

\begin{definition} 
A representation $X$ of $\Gamma$ for which all maps $\Psi(X,i)$, for $1 \le i \le m$, are invertible will be called a \emph{balanced representation}.
\end{definition}

 Let 
 $\mathcal{C}$ be the full subcategory of $\rep(\Gamma)$ consisting of the balanced representations. The following is immediate by dimension considerations: 
 
 \begin{lemma}
 \label{lemm:C-balanced}
The dimension vector of every balanced representation is balanced.
 \end{lemma}
 
 However, it is worth noting that it is not true in general that a representation with a balanced dimension vector is balanced. Now, we have the following categorical equivalence:

\begin{proposition} \label{PropFunctorF}
The functor $F$ induces an equivalence $F: \rep(A) \to \mathcal{C}$. In particular, the category $\mathcal{C}$ is abelian.
\end{proposition}

\begin{proof}
  Let $M \in \rep(A)$ and let $X = F(M)$. Recall that $E_{ij}$ denotes the identity of $A_{ij}$. By the definition of $F$, for all $i, j$ the map $\oplus_{k=1}^{w_{ij}} X(\alpha_{ij}^k)$ is nothing but the inclusion of $$\sum_{k=1}^{w_{ij}}e_{ij}^{kk}M = E_{ij}M$$ into $M$. It follows that the map $\Psi(X,i)$
  is the inclusion of $M$ into $M$, which is an isomorphism. This proves that $F$ is a functor from $\rep(A)$ into $\mathcal{C}$.
  
  To prove that $F$ is an equivalence, we construct a quasi-inverse $G: \mathcal{C} \to \rep(A)$ to $F$. Let $X \in \mathcal{C}$. To specify an $A$-module $G(X)$, we first let $G(X)=X_0$ as a vector space. To make $X_0$ an $A$-module, it is equivalent to make $X_0$ an $A_i$-module for each $1\le i\le m$ as discussed in the introduction. To do so, we
  specify the action of $e_{ij}^{pq}$ on $X$ for all $1\le j\le r_i$ and all $1\le i,j\le w_{ij}$ as follows. Let $x \in X_0$ and fix $i$. Since $X \in \mathcal{C}$, there are unique elements $x_{ijk} \in X_{ij}$ for all $1\le j\le r_i$ and $1\le k\le w_{ij}$ such that $$x = \sum_{j=1}^{r_i}\sum_{k=1}^{w_{ij}} X(\alpha_{ij}^k)(x_{ijk})$$
  We let $e_{ij}^{pq}x = X(\alpha_{ij}^p)(x_{ijq})$. Note that \begin{equation}
      \label{eq:x}
  e_{ij'}^{p'q'}(e_{ij}^{pq}x) = \delta_{jj'}\delta_{q'p}e_{ij}^{p'q}x=(e_{ij'}^{p'q'}e_{ij}^{pq})x,
  \end{equation}
  therefore $G(X)$ is a left $A_i$-module, as desired.
  
  To define $G$ on an arbitrary morphism $f=(f_v)_{v\in \Gamma_0}: X \to Y$ in $\mathcal{C}$, we simply take $G(f)$ to be the linear map $f_0: X_0\to Y_0$. Let $x\in X_0$ and let $x_{ijk}$ be as in the last paragraph for all $i,j,k$. Then for each $i,j,p,q$, we have 
    \begin{equation}
      \label{eq:f_0(y)}
  f_0(X(\alpha_{ij}^p)(x_{ijq})) = Y(\alpha_{ij}^p)(f_{ij}(x_{ijq}))
  \end{equation}
  since $f$ is a morphism in $\mathcal{C}$. 
  It then follows from Equation \eqref{eq:x} that
  \begin{equation}
      \label{eq:f_0(x)}
      f_0(x) = \sum_{j=1}^{r_i}\sum_{k=1}^{w_{ij}} Y(\alpha_{ij}^k)(f_{ij}(x_{ijk})).
  \end{equation}
  For $1 \le i \le m, 1 \le j \le r_i$ and $1 \le k \le w_{ij}$, Equation \eqref{eq:f_0(y)} implies that 
  $$G(f)(e_{ij}^{pq}x) = f_0(X(\alpha_{ij}^p)(x_{ijq})) = Y(\alpha_{ij}^p)(f_{ij}(x_{ijq}))$$
  and Equation \eqref{eq:f_0(x)} implies that 
  $$e_{ij}^{pq}G(f)(x)=e_{ij}^{pq}f_0(x)=Y(\alpha_{ij}^p)(f_{ij}(x_{ijq})),$$
  therefore $G(f)$ commutes with the action of $e_{ij}^{pq}$, so that $G(f)$ is indeed a morphism of $A$-modules. 
  
  We have specified the functor $G: \rep(\Gamma)\to \rep(A)$, so it remains to check that $G$ is a $\K$-linear functor and $G$ is a quasi-inverse to $F$. This is straightforward to check using the relevant definitions. 
\end{proof}

An abelian category $\mathcal{A}$ is \emph{hereditary} if the bifunctor $\Ext^2(-,-)$ vanishes, or equivalently, the covariant and contravariant functors $\Ext^1(M,-)$ and $\Ext^1(-,M)$ are right exact for every module $M$ in $\mathcal{A}$. The category $\rep(\Gamma)$ is hereditary, which implies that any abelian extension-closed subcategory of it is again hereditary. 

\begin{corollary}
The category $\rep(A)$ is hereditary.
\end{corollary}

\begin{proof}
 It suffices to check that $\mathcal{C}$ is closed under extensions. Let
$$0 \to L \to M \to N \to 0$$
be a short exact sequence with $L,N \in \mathcal{C}$. For each vertex $x$ in $\Gamma_0$, one can choose a suitable basis for $M_x = L_x \oplus N_x$ in such a way that $M(\alpha_{ij}^k)$ is given by a block diagonal matrix where the diagonal blocks are $L(\alpha_{ij}^k)$ and $ N(\alpha_{ij}^k)$ for all $i,j,k$. Since $L,N \in \mathcal{C}$, the maps $\Psi(L,i)$ and $\Psi(N,i)$ are invertible for all $i$. It follows that $\Psi(M,i)$ is invertible for all $i$ as well, therefore $M$ is in $\mathcal{C}$.
\end{proof}

\subsection{Varieties and general linear groups}Let $d\in (\Z_{\ge 0})^{\Gamma_0}$. We consider the affine space $\rep(\Gamma, d)$  given by
$$\rep(\Gamma, d) = \bigoplus_{\alpha:u \to v}{\rm Mat}_{d_v \times d_u}(\K)$$
where the direct sum is indexed by all arrows of $\Gamma$ and ${\rm Mat}_{d_v \times d_u}(\K)$ denotes the space of all $d_v \times d_u$ matrices over $\K$. This affine space parameterizes the representations of $\Gamma$ with dimension vector $d$, with each point $x$ in the space corresponding to the representation $M_x$ where $M_x(v) = \K^{d_v}$ for every vertex $v$ and $M(\alpha)$ is given by the corresponding matrix in ${\rm Mat}_{d_v \times d_u}(\K)$, in the canonical bases, for every arrow $\alpha: u \to v$. In the sequel, we identify each point  $x$ in $\rep(\Gamma, d)$ with the representation $M_x$. The group $${\rm GL}_d(\K) = \prod_{v \in \Gamma_0}{\rm GL}_{d_v}(\K)$$
acts on $\rep(\Gamma,d)$ by simultaneous conjugation via the map 
\[
{\rm GL}_d(\K)\times \rep(\Gamma, d)\to \rep(\Gamma, d),\quad (g, M)\mapsto g\cdot M 
\]
where $(g\cdot M)_\alpha=g_v M_\alpha g_u^{-1}$ for every arrow $\alpha:u\to v$ in $\Gamma$. The orbit of a representation under this action is precisely the isomorphism class of that representation in $\rep(\Gamma, d)$.

For each $d\in(\Z_{\ge 0})^{\Gamma_0}$, we let $\mathcal{C}_d$ denote the intersection $\mathcal{C} \cap \rep(\Gamma,d)$. Then $\mathcal{C}_d$ is nonempty only when $d$ is balanced, and when this is the case $\mathcal{C}_d$ is a Zariski-open set of $\rep(\Gamma,d)$. Let $F:\rep(A)\to\mathcal{C}$ and $G:\mathcal{C}\to\rep(A)$ be the quasi-inverse functors defined in \S \ref{sec:FandG}, and let  
$$\rep(A,d): = G(\mathcal{C}_d).$$
Then by Proposition \ref{PropFunctorF}, the representation space $\rep(A,d)$ may be identified with the affine variety $\mathcal{C}_d$ equipped with the action of the algebraic group
$${\rm GL}_d(\K) = \prod_{v \in \Gamma_0}{\rm GL}_{d_v}(\K).$$
Later, we will see how the ${\rm GL}_d(\K)$-variety $\mathcal{C}_d$ can be used to parametrize the simple $A$-modules in $\rep(A,d).$

\section{Stability}

We consider the generalized subspace quiver $\Gamma=\Gamma_A$ as constructed in the previous section. Let $\sigma \in \Hom_{\mathbb{Z}}(\mathbb{Z}^{\Gamma_0}, \mathbb{Z})$, which can also be thought of as a vector in $\mathbb{Z}^{\Gamma_0}$. If $\sigma = (\sigma_i)_{i \in \Gamma_0}$ and $d = (d_i)_{i \in \Gamma_0}$ is another  vector in $\Z^{\Gamma_0}$, we define $\sigma(d)$ to be the canonical dot product
$$\sigma(d) = \sum_{i \in \Gamma_0}\sigma_id_i.$$
We think of $\sigma$ as a \emph{stability parameter} for $\Gamma$. We need the following definition, which is King's interpretation of Mumford's notion of (semi)stability; see \cite{Ki}.

\medskip

\begin{definition}
\label{def:stability}
A representation $M \in\rep(\Gamma)$ is \emph{$\sigma$-semistable} if $\sigma(\dim (M))=0$ and  $\sigma(\dim(M'))\le 0$ for every subrepresentation $M'$ of $M$. If in addition we have $\sigma(\dim(M')) < 0$ for all proper and non-trivial subrepresentations of $M$, then $M$ is \emph{$\sigma$-stable}.
\end{definition}

This leads to the following.

\begin{definition}
\label{def:stable_vec}
   An element $d\in \Z_{\ge 0}^{\Gamma_0}$ is called \emph{$\sigma$-stable} or \emph{$\sigma$-semistable} if there is a $\sigma$-stable or $\sigma$-semistable representation with dimension vector $d$, respectively.
\end{definition}

We remark that if $M$ is $\sigma$-(semi)stable of dimension vector $d$, then there is a non-empty Zariski-open subset in $\rep(\Gamma,d)$ consisting of $\sigma$-(semi)stable representations; see \cite{Ki}.

\medskip

We will be interested in the particular stability parameter $\theta=(\theta_v)_{v\in \Gamma_0}$ where $\theta_0=-m$ and $\theta_{ij}=w_{ij}$ for all $1\le i \le m, 1\le j\le r_i$. The next three propositions discuss the close connections between this parameter and the category $\mathcal{C}$.

\begin{proposition} \label{FunctorImageSemiStable}
Let $X$ be a balanced representation. 
\begin{enumerate}
    \item We have $\theta(\dim(X))=0$.
    \item Let $Y$ be a subrepresentation of $X$ in $\rep(\Gamma)$. Then $\theta(\dim(Y))\le 0$, where equality holds if and only if $Y$ is balanced.
    \item The representation $X$ is $\theta$-semistable.
\end{enumerate}
\end{proposition}

\begin{proof}
Let $d=\dim(X)$ and $h=\dim(Y)$. Recall the definition of the maps $\Psi(X,i)$ and $\Psi(Y,i)$ for each $i$ from Section \ref{sec:FandG}. 
Since $X\in \mathcal{C}$, the map $\Psi(X,i)$ is an isomorphism for each $i$. Taking dimensions of the domain and codomain of $\Psi(X,i)$ then yields that $\sum_{j=1}^{r_i}w_{ij}d_{ij}=d_0$ for all $i$, therefore $$\theta(d) = -md_0 + \sum_{i=1}^m\sum_{j=1}^{r_i}w_{ij}d_{ij}=-mh_0 + mh_0 = 0,$$
proving Part (1). Next, note that $\Psi(Y,i)$ is a restriction of the isomorphism $\Psi(X,i)$ and hence injective for each $i$. Dimension considerations now imply that 
$\sum_{j=1}^{r_i}w_{ij}h_{ij} \le h_0$ for each $i$ and hence
$$\theta(h) = -mh_0 + \sum_{i=1}^m\sum_{j=1}^{r_i}w_{ij}h_{ij} \le -mh_0 + mh_0 = 0.$$ Here, for the equality to hold we must have $\sum_{j=1}^{r_i}w_{ij}h_{ij} = h_0$ for all $i$, which forces $\Psi(Y,i)$ to be an isomorphism for all $i$ and $Y$ to be in $\mathcal{C}$. Part (2) follows. Part (3) is immediate from parts (1) and (2). 
\end{proof}

\begin{example}
While every balanced representation is $\theta$-semistable by Proposition \ref{FunctorImageSemiStable}, the converse is not true. In Figure \ref{ce2}, we consider the quiver of the algebra $\K^3 * \K^3$, which has two clusters of size three with trivial weights. The stability parameter $\theta$ is  $\theta = (-2,1,\dots,1)$, where the first entry corresponds to vertex $v_0$. By an exhaustive check of all its subrepresentations, we can see that the given representation is $\theta$-stable. However, the representation is not balanced since the $3\times 3$ matrix formed by the maps from the right cluster is not invertible. 

\end{example}

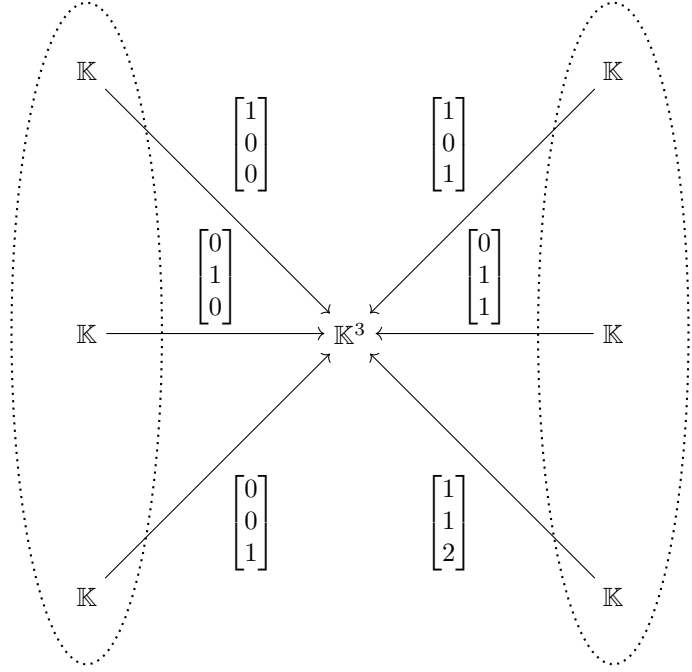
\begin{figure}\label{ce2}
    \centering
\begin{tikzpicture} 
[node distance={35mm}, main/.style = {}] 

\node (1) {$\mathbb{K}^3$};

\node[main] (2) [left of=1] {$\mathbb{K}$}; 
\node[main] (3) [above of=2] {$\mathbb{K}$}; 
\node[main] (4) [below of=2] {$\mathbb{K}$}; 
\node[main] (5) [right of=1] {$\mathbb{K}$}; 
\node[main] (6) [above of=5] {$\mathbb{K}$}; 
\node[main] (7) [below of=5] {$\mathbb{K}$}; 

\draw[->] (2) to [out=0,in=180,looseness=2] node[midway, above]{$\begin{bmatrix}0\\1\\0\end{bmatrix}$} (1);
\draw[->] (3) to [out=-45,in=135,looseness=2]  node[midway, above right]{$\begin{bmatrix}1\\0\\0\end{bmatrix}$} (1);
\draw[->] (4) to [out=45,in=-135,looseness=1.2]  node[midway, below right]{$\begin{bmatrix}0\\0\\1\end{bmatrix}$} (1);

\draw[->] (5) to [out=180,in=0,looseness=2] node[midway, above]{$\begin{bmatrix}0\\1\\1\end{bmatrix}$} (1);
\draw[->] (6) to [out=-135,in=45,looseness=2]  node[midway, above left]{$\begin{bmatrix}1\\0\\1\end{bmatrix}$} (1);
\draw[->] (7) to [out=135,in=-45,looseness=1.2]  node[midway, below left]{$\begin{bmatrix}1\\1\\2\end{bmatrix}$} (1);

\draw[dotted,thick] (3.5,0) ellipse (1 and 4.4);
\draw[dotted,thick] (-3.5,0) ellipse (1 and 4.4);

\end{tikzpicture}

    \caption{A $\theta$-semistable representation that is not balanced.}
\end{figure}

\begin{proposition}\label{Prop: simple to stable} Let $M$ be a representation in $\rep(A)$. Then $F(M)$ is $\theta$-stable if and only if $M$ is simple.
\end{proposition}

\begin{proof}
  Suppose $M$ is not simple. Then $M$ has a proper non-trivial subrepresentation $L$. This yields a proper subrepresentation $F(L)$ of $F(M)$. Since $F(L)$ is $\theta$-semistable by Proposition \ref{FunctorImageSemiStable}, the representation $F(M)$ cannot be $\theta$-stable by Definition \ref{def:stability}. It follows that $F(M)$ is $\theta$-stable only if $M$ is simple.

  Conversely, suppose $X=F(M)$ is not $\theta$-stable. By Proposition \ref{FunctorImageSemiStable}, we know that $X$ is $\theta$-semistable. It follows from Definition \ref{def:stability} that $X$ must have a proper, non-trivial subrepresentation $Y$ that is $\theta$-semistable. In particular, we have $\theta(\dim(Y))=0$, which implies $Y$ lies in $\mathcal{C}$ by Proposition \ref{FunctorImageSemiStable}.(2). Applying the quasi-inverse $G$ of $F$ to the inclusion $Y \subset F(M)$ in $\mathcal{C}$, we see that $M \cong GF(M)$ has a proper submodule isomorphic to $G(Y)$, so $M$ is not simple. We conclude that $F(M)$ is $\theta$-stable if $M$ is simple. 
\end{proof}

\begin{proposition} \label{subrep}
   Let $M \in \mathcal{C}$ and $N$ be a subrepresentation or a quotient of $M$. Then $\theta(\dim(N))=0$ if and only if $N$ is balanced.
\end{proposition}

\begin{proof} The ``if'' implication follows from Proposition \ref{FunctorImageSemiStable}.(1), and the ``only if'' implication follows from Proposition \ref{FunctorImageSemiStable}.(2) in the case where $N$ is a subrepresentation of $M$, so it remains to prove that whenever $N$ is quotient of $M$ with $\theta(\dim(N))=0$ we must have $N\in \mathcal{C}$. Suppose $N$ is such a quotient. Without loss of generality, we may assume that $N$ is the codomain of a surjective homomorphism $\pi: M\to N$. Consider the kernel $\Ker(\pi)\subseteq M$ of $\pi$ and the short exact sequence $0\to \Ker(\pi)\to M\to N\to 0$. Since $\theta(\dim(M))=\theta(\dim(N))=0$ by assumption, we must have $\theta(\dim(\ker(\pi)))=0$ by linear algebra, therefore $\Ker(\pi)\in \mathcal{C}$ by Proposition \ref{FunctorImageSemiStable}.(2). Since $\mathcal{C}$ is abelian and $N$ is isomorphic to the cokernel of the inclusion of $\Ker(\pi)$ into $M$, it follows that $N \in \mathcal{C}$, as desired. 
\end{proof}

The parameter $\theta$ is also connected to balanced dimension vectors (Definition \ref{def:balanced}): 
\begin{proposition}
\label{prop:balanced-semistable}
Let $d\in \Z_{\ge 0}^{\Gamma_0}$. If $d$
    is balanced, then $d$ is $\theta$-semistable.  
\end{proposition}
\begin{proof}
  Let $d\in (\Z_{\ge 0})^{\Gamma_0}$ be balanced.
  By Proposition~\ref{FunctorImageSemiStable}, it suffices to show that there is a representation $M \in \mathcal{C}$ with dimension vector $d$.
  Thus we must find a representation $M$ of dimension $d$ such that for each cluster $i$ the map  
  $$\Psi(M,i) = \bigoplus_{j=1}^{r_i}\bigoplus_{k=1}^{w_{ij}} M(\alpha_{ij}^k): \left(\bigoplus_{j=1}^{r_i}\bigoplus_{k=1}^{w_{ij}} M_{ij}\right) \longrightarrow M_0$$
is an isomorphism.
  
Define $M(v_0) = \K^{d_0}$ and let $M(v_{ij}) = \K^{d_{ij}}$ for each $i=1,2,\dots,m$ and $j=1,,2\dots,r_i$.
Since $d$ is balanced, we have $\sum_{j=1}^{r_i}w_{ij}d_{ij}=d_0$ for each $i$, so we may pick a vector space isomorphism $f_i:\left(\oplus_{r=1}^{r_i}\oplus_{k=1}^{w_{ij}}M_{ij}\right)\to M_0$ for each $i$. Defining $M(\alpha_{ij}^k)$ as the suitable restriction of $f_{i}$ to $M_{ij}$ for all $i,j,k$ yields a representation $M$ with the desired properties.
\end{proof}

\section{Stable dimension vectors}

In order to find the finite dimensional simple $A$-modules, we first need to determine for what balanced vectors $d\in (\Z_{\ge 0})^{\Gamma_0}$ the set $\rep(A,d)$ contains a simple $A$-module. By Proposition \ref{Prop: simple to stable} and Definition \ref{def:stable_vec}, we may do so by finding $\theta$-stable vectors in $(\Z_{\ge 0})^{\Gamma_0}$ for our stability parameter $\theta$.

\subsection{Two rational polyhedral cones}

Given an acyclic quiver $Q$ and a stability parameter $\sigma\in\mathbb{Z}^{Q_0}$, Derksen and Weyman \cite{DW} give a characterization of the $\sigma$-stable dimension vectors of $Q$. We will adapt their result to our quiver $\Gamma$ and to our stability parameter $\sigma = \theta$ defined previously. To this end, we let $\Sigma(\Gamma,\theta)$ denote the set of $\theta$-semistable elements of $(\Z_{\ge 0})^{\Gamma_0}$ and let $\mathbb{Q}_+\Sigma(\Gamma,\theta)$ denote the rational cone generated by these elements in $\R^{\Gamma_0}$. It is well known that this is a rational polyhedral cone. 

We will also be interested in balanced vectors (Definition \ref{def:balanced}). Let $\Sigma_B$ be the set of balanced vectors in $(\Z_{\ge 0})^{\Gamma_0}$ and let $\mathbb{Q}_+\Sigma_B$ denote the corresponding rational cone in $\R^{\Gamma_0}$. 
Since the requirements for a vector to be balanced are given by
a set of linear equations, it is clear that $\mathbb{Q}_+\Sigma_B$ is also a rational polyhedral cone. By Proposition \ref{prop:balanced-semistable}, the set $\Sigma_B$
is a subset of $\Sigma (\Gamma,\theta)$, so $\mathbb{Q}_+\Sigma_B$ is a subcone of $\mathbb{Q}_+\Sigma (\Gamma,\theta)$.
To use the results of \cite{DW}, we need the following notion:

\begin{definition}
   Let $\mathfrak{C}$ be a rational polyhedral cone and let $\alpha \in \mathfrak{C}$. A \emph{minimal conical decomposition for $\alpha$ in $\mathfrak{C}$} is a decomposition $$\alpha = \sum_{i=1}^sa_i\delta_{i}$$ where \begin{enumerate}
        \item The number $a_i$ is positive and rational for every $i$;
        \item The $\delta_{i}$'s are linearly independent vectors all lying on extremal rays.
    \end{enumerate}
\end{definition}

By (the conical version of) Carathéodory's theorem, minimal canonical decompositions always exist in rational polyhedral cones. We give a proof for completeness.

\begin{lemma}[Carathéodory] Let $\mathfrak{C}$ be a rational polyhedral cone and let $\alpha \in \mathfrak{C}$. Then a minimal canonical decomposition exists for $\alpha$. 
\end{lemma}

\begin{proof}
Let $\alpha = \sum_{i = 1}^t a_i \delta_i$ be a (rational) positive linear combination of extremal vectors with a minimal number of summands. 
We claim that $\delta_1, \ldots, \delta_t$ are linearly independent. If not, 
there is a rational dependence relation $\sum_{i=1}^t b_i\delta_i =0$ where at least one $b_i$ is positive.  Let $\mu := \min\{a_i/b_i \mid b_i > 0\}=a_j/b_j$.  Then $a_i - \mu b_i \geq 0$ for all $i$ and $$\alpha = \sum_{i=1}^t (a_i- \mu b_i)\delta_i.$$
This yields a decomposition of $\alpha$ as a positive linear combination of extremal vectors having fewer than $t$ nonzero summands, which is a contradiction.
\end{proof}

\subsection{Extremal rays of $\mathbb{Q}_+\Sigma_B$}
We will show that the cone $\mathbb{Q}_+\Sigma_B$ has precisely $\prod_{i=1}^mr_i$ extremal rays, one for each way of choosing one vertex within each cluster.

\begin{definition}
We let $U=\{u=(u_1,..., u_m)\in \Z^m : 1\le u_i\le r_i\text{ for all } i\}$, and
for each element $u = (u_1, u_2, \ldots, u_m)\in U$, we define a vector $\delta_u\in \mathbb{Q}_{\ge 0}^{\Gamma_0}$ with $(\delta_u)_0=1$ and $$(\delta_u)_{ij} = \begin{cases} \frac{1}{w_{iu_i}} &\; \text{if} \, j=u_i,\\ 0& \; \text{otherwise.} \end{cases}$$
for all $i,j$. Furthermore, we let $E$ denote the collection of these $\prod_{i=1}^mr_i$ vectors.
\end{definition}

Figure \ref{extremalraypic} shows an example of one of these vectors.

 \begin{proposition}
 \label{lem extremal rays}
 In the above notation, the set $E$ consists of one vector on each extremal ray of the cone $\mathbb{Q}_+\Sigma_B$.
 \end{proposition}

 \begin{proof}
First of all, it is clear that elements from $E$ are balanced. Now, it follows from Proposition \ref{prop:balanced-semistable} that the integral vectors on the rays from $E$ are $\theta$-semistable, hence are in $\mathbb{Q}_+\Sigma_B$.

 We now check that the set $E$ contains all of the extremal rays of $\mathbb{Q}_+\Sigma_B$. 
 To do so, it suffices to show that every nonzero element $d \in \mathbb{Q}_+\Sigma_B$ can be written as a positive rational combination of the elements of $E$. We show this by induction on the size $l$ of the support $\{(i,j) \mid d_{ij} \neq 0\}$ of $d$. Note that $l\ge m$ since $d$ is balanced and nonzero, and that if $l=m$ then we must have $d=d_0\cdot \delta_u$ for some $u\in U$.  Now suppose  $l>m$. 
  Let $(i_0,j_0)$ be a pair such that $d_{i_0 j_0} w_{i_0 j_0}$ is the minimum non-zero value in the set $\{d_{ij} w_{ij} \mid 1 \leq i \leq m, 1\leq j \leq r_i\}$.
    For each cluster $i$, choose
  $u_i$ with $1 \leq u_i \leq r_i$ in such a way that that $d_{iu_i} > 0$ for all $i$ and $u_{i_0}=j_0$.  Let $u=(u_i)_{1\le i\le m}$, and set $d' = d - d_{i_0 j_0}w_{i_0 j_0} \delta_u$.
  Then $d'$ is a balanced vector in $(\Z_{\ge 0})^{\Gamma_0}$, and we have $(d')_{i_0 j_0}=0$, therefore
  the support of $d'$ is properly contained in the support of $d$.  
   By induction, we can write $d'$ as a positive rational combination of the elements of $E$, so we can do the same for 
  $d = d'+ (d_{i_0 j_0}w_{i_0 j_0})\delta_u$, as desired. Finally, the minimality of $E$ follows from support considerations: an element of $E$ cannot be written as a positive linear combination of other elements of $E$. It follows from this that all elements of $E$ are extremal.
 \end{proof}

\begin{remark}
\label{rmk:dimension}
By the proof of the last proposition, every element in $\mathbb{Q}_+\Sigma_B$ has a canonical combination into at most $(\sum_{i=1}^m r_i) + 1-m$ extremal vectors, where this expression is the total number of vertices of $\Gamma$ minus $m$. This just follows from the inductive argument of that proof: each time we remove a term, the support decreases by one; and when we reach a support that is non-zero at a single vertex of each cluster, the resulting vector is a positive multiple of the corresponding element of $E$.  On the other hand, it is not hard to show that there is a subset of $E$ of size $(\sum_{i=1}^mr_i) + 1-m$ that is linearly independent, by support considerations.
Therefore, the dimension of the cone $\mathbb{Q}_+\Sigma_B$ is $(\sum_{i=1}^mr_i) + 1-m$. 
\end{remark}

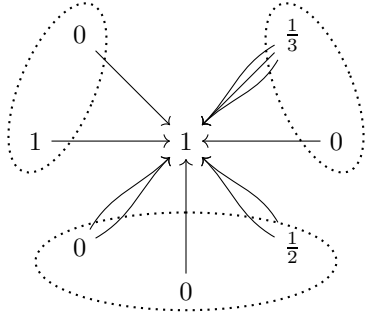
\begin{figure}
\label{extremalraypic}
    \centering
\begin{tikzpicture} 
[node distance={20mm}, main/.style = {}] 

\node (1) {$1$};

\node[main] (2) [below left of =1] {$0$}; 
\node[main] (3) [below of=1] {$0$}; 
\node[main] (4) [below right of =1] {$\frac{1}{2}$}; 

\node[main] (5) [right of=1] {$0$}; 
\node[main] (6) [above right of=1] {$\frac{1}{3}$}; 

\node[main] (7) [left of=1] {$1$}; 
\node[main] (8) [above left of=1] {$0$};

\draw[->] (2) to [out=60,in=-135,looseness=1] (1);
\draw[->] (2) to [out=30,in=-135,looseness=1] (1);
\draw[->] (4) to [out=120,in=-45,looseness=1] (1);
\draw[->] (4) to [out=150,in=-45,looseness=1] (1);
\draw[->] (3) to [out=90,in=-90,looseness=1] (1);

\draw[->] (5) to [out=180,in=0,looseness=1] (1);
\draw[->] (6) to [out=-135,in=45,looseness=1] (1);
\draw[->] (6) to [out=-120,in=45,looseness=1] (1);
\draw[->] (6) to [out=-150,in=45,looseness=1] (1);

\draw[->] (7) to [out=0,in=180,looseness=1] (1);
\draw[->] (8) to [out=-45,in=135,looseness=1] (1);

\draw[dotted,thick,rotate=22.5] (1.85,0) ellipse (0.5 and 1.2);
\draw[dotted,thick,rotate=-22.5] (-1.85,0) ellipse (0.5 and 1.2);

\draw[dotted,thick,rotate=90] (-1.6,0) ellipse (0.65 and 2);

\end{tikzpicture}
\caption{An example of an extremal ray of the subcone of balanced dimension vectors.}
\end{figure}

\medskip

The Euler bilinear form will play a crucial role in our investigation.

\begin{definition}
Let $d$ and $d'$ be in $\mathbb{Q}^{\Gamma_0}$. The \emph{Euler bilinear form} $\langle , \rangle$ is defined by 
$$ \langle d, d' \rangle = \sum_{i \in \Gamma_0} d_id'_i - \sum_{\alpha: u \to v} d_{u}d'_{v}$$
where the second sum runs through the arrow set of $\Gamma$.
\end{definition}

We recall the representation-theoretic interpretation. If $M,N$ are two representations, then
$$\langle {\rm dim}(M), {\rm dim}(N) \rangle = {\rm dim}{\rm Hom}(M,N) - {\rm dim Ext}^1(M,N).$$
Recall that a dimension vector $d$ is a \emph{Schur root} if there is a representation $M$ of dimension vector $d$ with ${\rm End}(M) = \K$. A $\theta$-stable dimension vector is always a Schur root; see \cite{DW}. Furthermore, a dimension vector $d$ is \emph{real} if $\langle d, d \rangle = 1$, \emph{isotropic} if $\langle d,d\rangle = 0$ and \emph{imaginary} if $\langle d, d \rangle < 0$. 

\medskip

Suppose $d$ is balanced. Then for any $\delta_u$, we get 
\begin{equation}
    \label{eq:d-delta_u}
\langle d,\delta_u \rangle = n + \sum_{i=1}^m \frac{d_{iu_i}}{w_{iu_i}}-  \sum_{i=1}^m\sum_{j=1}^{r_i}{d_{ij}w_{ij}}=\sum_{i=1}^m\frac{d_{iu_i}}{w_{iu_i}} - n(m-1)
\end{equation}
as well as
$$\langle \delta_u,d \rangle = n + \sum_{i=1}^m\frac{d_{iu_i}}{w_{iu_i}}-  \sum_{i=1}^mn=\sum_{i=1}^m\frac{d_{iu_i}}{w_{iu_i}} - n(m-1).$$ 

\medskip

\par Moreover, from this calculation, we can see that $$\langle \delta_u, \delta_v \rangle = \sum_{i|u_i = v_i}^m \frac{1}{w_{iu_i}^2} - (m-1).$$

Here are some additional properties of the extremal rays of $\mathbb{Q}_+\Sigma_B$ with respect to the Euler bilinear form. The proofs are just straightforward calculations.

\begin{proposition} \label{PropertiesExtremalVectors}
Let $\delta_u, \delta_v$ be distinct, and let $d, e$ be balanced vectors. Then we have the following.
\begin{enumerate}[$(1)$]
\item $\langle d,e\rangle = \langle e,d \rangle $.
    \item $\langle \delta_u, \delta_v \rangle \leq 0$, with equality if and only if both $u_i = v_i$ for all but one $i$ and when $u_i = v_i$, we have $w_{i,u_i}=1$.
    \item We have $\langle \delta_u, \delta_u \rangle < 0$ if and only if at least two weights $w_{iu_i}$ are not 1.
    \item We have $\langle \delta_u, \delta_u \rangle > 0$ if and only if at most one weight $w_{iu_i}$ is not 1.
\end{enumerate}
\end{proposition}

Note that Proposition \ref{PropertiesExtremalVectors} yields that no dimension vector on an extremal ray of $\mathbb{Q}_+\Sigma_B$ is isotropic. Moreover, any minimal vector parallel to some $\delta_u$ which satisfies $(4)$ has to be a real Schur root, since a $\theta$-stable dimension vector is necessarily a Schur root. The following definition is found in \cite{DW}. Remark that if $\delta_1, \delta_2$ are two distinct $\theta$-stable dimension vectors, then $\langle \delta_1, \delta_2 \rangle \le 0$. Indeed, if $M_i$ is a representation of dimension vector $\delta_i$ which is in general position, then $\Hom(M_1, M_2)=0$, so that
$$\langle \delta_1, \delta_2 \rangle = {\rm dim}\Hom(M_1, M_2) - {\rm dim}{\rm Ext}^1(M_1, M_2) \le 0.$$

  \begin{definition}
     Given a set $D=\{d_1,d_2,\dots,d_r\}$ of balanced $\theta$-stable dimension vectors, we define the quiver $Q(D)$ whose vertices are $Q(D)_0 = \{1,2,\dots,r\}$.
     The arrows of $Q(D)$ are determined as follows.  
     There are $1-\langle d_i,d_i\rangle$ loops at $i$ and there are $-\langle d_i,d_j \rangle$ arrows from $i$ to $j$.
  \end{definition}
  
 A quiver is \emph{path-connected} if any two vertices $i$ and $j$ in it can be connected by a (directed) path from $i$ to $j$ and also by a path from $j$ to $i$. Recall that a quiver is \emph{connected} is the underlying graph (forgetting orientation of arrows) is connected. Clearly, any quiver that is path-connected is connected, but the converse is not true. Indeed, any quiver $Q$ whose underlying graph is a non-trivial tree is connected but not path-connected.  Since, by Proposition \ref{PropertiesExtremalVectors}, the Euler bilinear form is symmetric for the balanced dimension vectors, it follows that the quiver $Q(D)$ defined above is connected if and only if it is path connected. 
  In what follows, a dimension vector is said to be \emph{divisible} if its entries share a non-trivial common factor. Otherwise, it is \emph{indivisible}. The following theorem is an adaptation of Theorem 6.4 of \cite{DW}.
  
  \begin{theorem}\label{dw 6.4}
  Let $d$ be a dimension vector in $\mathbb{Q}_+\Sigma_B$ with minimal conical decomposition
  $$d = \sum_{i=1}^r a_i \delta_i$$
  where the $\delta_i$ are some extremal dimension vectors of  $\mathbb{Q}_+\Sigma_B$. Then $d$ is $\theta$-stable if and only if either
  \begin{enumerate}[$(1)$]
        \item $d$ is the minimal integral vector on an extremal ray of a $\delta_u$ whose support has at most one non-trivial weight, or
        \item $\langle d, \delta_u \rangle \leq 0$ for each extremal vector $\delta_u$, the quiver $Q(\{\delta_1, \ldots, \delta_r\})$ is (path) connected, and $d$ is not both isotropic and divisible.
    \end{enumerate}
  \end{theorem}
  \begin{proof}
  We start with some observations. First, note that $\langle d, \delta_u \rangle = \langle \delta_u, d \rangle$ for all extremal vectors $\delta_u$ of $\mathbb{Q}_+\Sigma_B$. Second, we have shown in Proposition \ref{PropertiesExtremalVectors} that no vector on an extremal ray of $\mathbb{Q}_+\Sigma_B$ is isotropic. Moreover, by Proposition \ref{PropertiesExtremalVectors}, part (4), the ray of $\delta_u$ contains a real Schur root if and only if the support of $u$ has at most one non-trivial weights. If $d$ is on the ray of $\delta_u$ and is imaginary, then the second condition in the statement is satisfied. Finally, if $\delta_v$ is any extremal vector with $\langle d, \delta_v \rangle > 0$, then (a multiple of) $\delta_v$ has to appear in the minimal conical decomposition for $d$. Indeed, this follows from the facts that the Euler form is bilinear and that for $\delta, \delta'$ distinct $\theta$-stable vectors, one has $\langle \delta, \delta' \rangle \le 0$.

  Now, a careful review of the proof of Theorem 6.4 (and Lemma 6.1) of \cite{DW} shows that it is still valid if one replaces ``extremal rays of $\mathbb{Q}_+\Sigma (\Gamma,\theta)$" by ``extremal rays of $\mathbb{Q}_+\Sigma_B$". Indeed, the only crucial property that is needed is that each $\delta_u$ has a non-zero multiple that is $\theta$-stable. Therefore, it follows from Theorem 6.4 of \cite{DW} that $d$ is $\theta$-stable if and only if either 
  \begin{enumerate}[$(i)$]
      \item $d$ is a real Schur root on an extremal ray of $\mathbb{Q}_+\Sigma_B$,
      \item or else that $\langle d, \delta_u \rangle \le 0, \langle \delta_u, d \rangle \le 0$ for all $u \in E$, $Q(\{\delta_1, \ldots, \delta_r\})$ is path-connected and $d$ is indivisible when $d$ is isotropic.
  \end{enumerate} 
 Now, it is straightforward to check, using the above observations, that $(1)$ is equivalent to $(i)$ and that $(2)$ is equivalent to $(ii)$.
  \end{proof}
  
 According to Theorem \ref{dw 6.4}, we need to study (path-)connectedness of the quivers of the form $Q(D)$ where $D$ is a collection of linearly independent $\delta_u$, and we also need to find those isotropic dimension vectors $d$ with the property that $\langle d, \delta_u \rangle \le 0$ for all extremal vector $\delta_u$ of $\mathbb{Q}_+\Sigma_B$. 
 We achieve those below. We begin our analysis by studying connectedness of the quiver $Q(D)$.

\begin{lemma}\label{Lem: isolated vertex}  Let $d$ be a balanced dimension vector with minimal conical decomposition
 $d = \sum_{i=1}^r  a_i \delta_i$ where each $a_i$ is a positive rational number. 
  If the quiver $Q(\{\delta_1, \ldots, \delta_r\})$ contains an isolated vertex $\ell$,  then $\langle d, \delta_\ell\rangle > 0$. 
\end{lemma}
\begin{proof}
 Consider $i$ with $1 \leq i \leq r$ and $i \neq \ell$.  
 Since $i$ is not adjacent to $\ell$ in $Q(\{\delta_1, \ldots, \delta_r\})$ it follows that $\langle \delta_i , \delta_\ell \rangle = 0$. From parts (2) and (4) of Proposition \ref{PropertiesExtremalVectors}, we know that all but one weights in the support of $\delta_\ell$ are trivial, which in turn implies that $\langle \delta_\ell, \delta_\ell \rangle > 0$.
   Thus $\langle d , \delta_\ell\rangle = \sum_{i=1}^r a_i\langle \delta_i , \delta_\ell\rangle = a_\ell \langle \delta_\ell , \delta_\ell\rangle > 0$.
\end{proof}

Next we consider the case where $Q(D)$ has no isolated vertices.

\begin{lemma}\label{lem:disconnected}
 Let $d$ be a balanced dimension vector with minimal conical decomposition
 $d = \sum_{i=1}^r  a_i \delta_i$ where each $a_i$ is a positive rational number.
  Suppose that the quiver $Q(\{\delta_1, \ldots, \delta_r\})$ has no isolated vertex.
 Then $Q(\{\delta_1, \ldots, \delta_r\})$ is connected.
\end{lemma}

\begin{proof}
  Let $D = \{\delta_1, \ldots, \delta_r\}$. Assume that $Q(D)$ is disconnected. Hence, it contains at least two connected components $C$ and $C'$. Since the quiver has no isolated vertices we have
  $|C|\geq 2$ and $|C'| \geq 2$.  Let $\delta_i, \delta_j$ be two  vertices in $C$ which are joined by an arrow.  Similarly let $\delta_p, \delta_q$ be two vertices in $C'$ which are joined by an arrow.

  Define $u_\ell \in \{(b_1,b_2,\dots,b_m) \mid 1 \leq a_i \leq r_i\}$ by $\delta_\ell=\delta_{u_\ell}$ for $1 \leq \ell \leq r$. 
  Without loss of generality we may assume that $u_i = (1,1,\dots,1)$.  Since
  $\delta_i$ and $\delta_p$ are not adjacent, we have $\langle \delta_i , \delta_p\rangle=0$, so by Proposition \ref{PropertiesExtremalVectors}.(2)
  we may assume, without loss of generality, that $u_p=(2,1,1,\dots,1)$ and that $w_{i1}=1$ for $i \ge 2$. Since $\delta_j$ is adjacent to $\delta_i$ but not adjacent to $\delta_p$, we may similarly assume that $u_j=(2,2,1,1,\dots,1)$ and that $w_{12}=1$. Finally, since $\delta_q$ is adjacent to $\delta_p$ but not adjacent to either $\delta_i$ nor $\delta_j$, it follows that $u_q=(1,2,1,1,\dots,1)$ and that $w_{11} = w_{22}= 1$ as well.
  
  Suppose there exists $\delta_k \in Q_0(D)\setminus C$.  Such a vertex $\delta_k$ is not adjacent to $\delta_i$ nor to $\delta_j$.
 This implies $u_k=u_p$ or $u_k=u_q$ and thus  $Q_0(D)\setminus C = \{\delta_p,\delta_q\}=C'$.  Similarly $Q_0(D)\setminus C' =C=\{\delta_i,\delta_j\}$.  Thus $D= C \cup C' = \{\delta_i,\delta_j,\delta_p,\delta_q\}$.
 Observe that $\delta_i + \delta_j = \delta_p + \delta_q$. This contradicts that the vectors in the minimal conical decomposition are linearly independent.
\end{proof}

Combining Lemmas \ref{Lem: isolated vertex} and \ref{lem:disconnected} yields the following consequence.

\begin{cor}\label{corollary: connected}
Let $d$ be a balanced dimension vector with minimal conical combination
 $d = \sum_{i=1}^r  a_i \delta_i$ where each $a_i$ is a positive rational number.
 If $\langle d, \delta_u\rangle\leq 0$
  for all extremal vector $\delta_u$ of $\mathbb{Q}_+\Sigma_B$,
 then $Q(\{\delta_1, \ldots, \delta_r\})$ is connected.\qed
\end{cor}

The following definition provides a family of balanced dimension vectors that are easy to identify, but that will need to be treated differently for the purpose of the main theorem.

 \begin{definition}
A balanced dimension vector $d$ of $\Gamma$ with $d_0 = n > 2$ is \emph{special isotropic} if $d$ has exactly $4$ entries $d_{ij}$ with $0 < d_{ij} < n$, and these entries are all equal to  $n/2$ with trivial corresponding weights.
\end{definition}
  
We note that for a special isotropic dimension vector, there are exactly two clusters with (unordered) entries $n/2, n/2, 0, \ldots, 0$ and all of the other clusters $i$ have a vertex with entry $n$ and trivial weight (all other entries in cluster $i$ are zero). It is not hard to check that such a dimension vector is isotropic, and is indivisible if and only if $n=2$.

\begin{lemma}\label{lem: isotropic}
If $d$ balanced isotropic and $\langle d,\delta_u\rangle\leq 0$ for each tuple $u\in E$, then $d$ is special isotropic. 
\end{lemma}
\begin{proof}
   We may assume that $d_{ij} \ne n$ for all $i,j$, as otherwise, we can reduce the problem to fewer clusters. Let $d=\sum_{i=1}^s a_i\delta_i$ be a minimal conical decomposition of $d$ and let $D=\{\delta_i:1\le i\le s\}$. Since $d$ is isotropic and $\langle d,\delta_u\rangle\leq 0$ for each tuple $u$, we must that $\langle d,\delta_v\rangle =0$ for each $\delta_v \in D$.
   
   Assume by contradiction there exists a cluster $i$ with two vertices $j$ and $j'$ such that $$\frac{d_{ij}}{w_{ij}}>\frac{d_{ij'}}{w_{ij'}}>0$$ 
   There must exist at least one  $\delta_{v} \in D$ whose support contains $v_{ij'}$.  Let $u$ be a tuple whose support agrees with $v$ everywhere except at cluster $i$ where it is $v_{ij}$. From Proposition \ref{PropertiesExtremalVectors}, it follows that
   $$\langle d, \delta_u \rangle>\langle d, \delta_v \rangle = 0$$ 
   This contradicts the fact that $\langle d, \delta_u \rangle \leq 0$ for every tuple $u$. Thus, the set $\{{d_{ij}}/{w_{ij}}: 1\le j\le r_i\}$ contains a unique nonzero value, say $q_i$, for all $i$. 
   Now, since $d$ is balanced, for each cluster $i$, we get
   \begin{equation}
       \label{eq:force}
   n=\sum_{j=1}^{r_i}w_{ij}d_{ij}=q_i\sum_{j\mid d_{ij} > 0}w_{ij}^2,
   \end{equation}
   so $q_i$ divides $n$. Observe that $q_i=n$ occurs only when $d_{ij}$ is nonzero for a unique $j$, and for this $j$, we must have $d_{ij}=n$ and $w_{ij}=1$, which we already assumed does not occur. Therefore, $q_i\leq \frac{n}{2}$ for every cluster $i$. Furthermore, for any $u$, $\langle d,\delta_u\rangle = 0$ implies that $\sum_{i=1}^m q_i = n(m-1)$ by Equation \eqref{eq:d-delta_u}. This implies that $n(m-1) \leq m\frac{n}{2}$, hence $m\leq 2$, from which we deduce that $1 \le m \le 2$. If $m=1$, then $q_1=0$, contradicting the fact that each $q_i$ is nonzero by definition. It follows that we must have $m=2$ and $q_1=q_2=n/2$. By Equation \eqref{eq:force}, this can happen only if $d$ is special isotropic. 
\end{proof}

\section{Simple representations of $A$}

In this section, we use the equivalence between $\rep(A)$ and the subcategory $\mathcal{C}$ together with the description of the $\theta$-stable balanced dimension vectors of $\Gamma$ to describe the finite dimensional simple $A$-modules.

In the following theorem, recall that the extremal vector $\delta_u$ is real if and only if the support of $u$ has at most one non-trivial weight. If the support of $u$ only has trivial weights, then $\delta_u$ is the minimal integral vector on the ray of $\delta_u$. If $u$ has support containing exactly one non-trivial weight $w$, then $w\delta_u$ is the minimal integral vector on the ray of $\delta_u$.

\begin{theorem} \label{mainThm}Let $n$ be a positive integer and let $d$ be a balanced dimension vector of $\Gamma$ with $d_0 = n$. Let $M$ is an $n$-dimensional representation of $A$ where $F(M)$ has dimension vector $d$, that is, $d_{ij}={\rm dim}_\K (E_{ij}M)/w_{ij}$. 
\begin{enumerate}
\item The dimension vector $d$ is $\theta$-stable if and only if either 
\begin{enumerate}
    \item $d$ is minimal on the ray of some real extremal ray $\delta_u$ or
    \item  $d$ is special isotropic indivisible, or
    \item $d$ is not special isotropic and
$$\sum_{i=1}^m{\rm max}(d_{i1}/w_{i1}, \ldots, d_{i r_i}/w_{i r_i}) \le (m-1)n.$$
\end{enumerate}
    \item If $M$ is simple, then $d$ is $\theta$-stable, hence satisfies (1).
    \item 
   If $d$ is $\theta$-stable, then the space of $n$-dimensional representations $M$ with $F(M)$ having dimension vector $d$ has a Zariski-dense and open subset of simple representations.
\end{enumerate}
\end{theorem}

\begin{proof}
Let us consider the first statement. It follows from Theorem \ref{dw 6.4} that $d$ is $\theta$-stable if and only if either
\begin{enumerate}[$(i)$]
        \item $d$ is the minimal integral vector on the extremal ray containing $\delta_u$ where the support of $\delta_u$ has at most one non-trivial weight, or
        \item $\langle d, \delta_u \rangle \leq 0$ for each extremal vector $\delta_u$, the quiver $Q(\{\delta_1, \ldots, \delta_r\})$ is (path) connected, and $d$ is not both isotropic and divisible.
    \end{enumerate}
Clearly, statement $(1)(a)$ is equivalent to $(i)$. Assume that $d$ is special isotropic indivisible.  In this case, a minimal conical decomposition for $d$ is given by $d = \delta_1 + \delta_2$ for two extremal vectors $\delta_1, \delta_2$ that are real Schur roots with $\langle \delta_i, \delta_j \rangle = -1$ for $1 \le i,j \le 2$. The quiver $Q(\{\delta_1, \delta_2\})$ has two vertices $\delta_1, \delta_2$ and two arrows $\delta_1 \to \delta_2$ and $\delta_2 \to \delta_1$ and hence is path-connected. To check that $\langle d, \delta_u \rangle \le 0$ for all extremal vector $\delta_u$, it suffices to prove this for $\delta_u \in \{\delta_1, \delta_2\}$. We have $\langle d, \delta_i \rangle=0$ for $1 \le i \le 2$. Hence, $(1)(b)$ implies $(ii)$. Now, assume that $(1)(c)$ holds. If $d$ is a real Schur root lying on an extremal ray, then $(i)$ is satisfied. If $d$ lies on an imaginary extremal ray, then $(ii)$ is satisfied.
So assume that $$\sum_{i=1}^m{\rm max}(d_{i1}/w_{i1}, \ldots, d_{i r_i}/w_{i r_i}) \le (m-1)n$$
while $d$ does not lie on an extremal ray and is not special isotropic. Then any minimal conical decomposition of $d$ has at least two summands. The given inequality ensures that $\langle d, \delta_u \rangle \le 0$ for all extremal vectors $\delta_u$. We have seen that the quiver of this conical decomposition is always path-connected in this case. As $d$ is not special isotropic, $(ii)$ is satisfied. 

Now, assume that $(ii)$ is satisfied. 
There is some $u$ such that $\langle d, \delta_u \rangle$ equals $\sum_{i=1}^m{\rm max}(d_{i1}/w_{i1}, \ldots, d_{i r_i}/w_{i r_i}) - n(m-1)$ by Equation \ref{eq:d-delta_u}, which gives $(1)(c)$. 

\medskip

Part $(2)$ follows from Proposition \ref{Prop: simple to stable}.

\medskip

Finally, for part $(3)$, when $d$ is $\theta$-stable, then $\rep(\Gamma,d)$ has a non-empty open set of $\theta$-stable representations; see \cite{Ki}. This translates to the desired statement for the space of $n$-dimensional representations $M$ of $A$ for which $F(M)$ has dimension vector $d$.
\end{proof}

\section{Modules in general position}
A representation $M$ in $\rep(\Gamma, d)$ is in \emph{general position} or is a \emph{general representation of dimension vector $d$} if $M$ belongs to a given proper open (hence dense) set of $\rep(\Gamma,d)$. When we say that a general representation of dimension vector $d$ satisfies a certain property, we mean by this that we can find a non-empty proper open set of $\rep(\Gamma, d)$ such that all representations in that open set satisfy the given property.

\medskip

We have proven that the dimension vector of any simple $A$-module is a $\theta$-stable balanced dimension vector and, conversely, given a $\theta$-stable balanced dimension vector $d$, a general representation of dimension vector $d$ is a simple object of $\mathcal{C}$, hence corresponds to a simple $A$-module. In this section, we analyse finite dimensional $A$-modules in general position, or equivalently, general representations of $\Gamma$ having balanced dimension vector.

We start with the following proposition, which guarantees the existence of non-semisimple representations, whenever we are working with a non-trivial free product.

\begin{proposition}
The algebra $A$ admits a non-semisimple finite dimensional representation if and only if $A$ is infinite dimensional.
\end{proposition}

\begin{proof}
 Clearly, $A$ is infinite dimensional if and only if $m > 1$, that is, $A$ has at least two (non-trivial) factors. If $A$ has a single factor $A_1$ which is semisimple, then all finite dimensional modules are semisimple. 
 Let us now assume that $A$ is infinite dimensional. If we have at least two non-trivial weights in distinct clusters, then there is a $\theta$-stable dimension vector $\delta$ lying on an extremal ray that is imaginary. Let $M$ be a general representation of dimension vector $\delta$. In particular, the endomorphism algebra of $M$ is trivial. Since $${\rm dim}{\rm Hom}(M,M) - {\rm dim}{\rm Ext}^1(M,M) = \langle \delta,\delta \rangle < 0,$$ 
we deduce that $M$ has self-extensions. Let
$$0 \to M \to E \to M \to 0$$
be a non-split self-extension. Since the endomorphism algebra of $M$ is trivial, the middle term $E$ has to be indecomposable. Since $M$ lies in $\mathcal{C}$, which is closed under extensions, $E$ lies in $\mathcal{C}$. We have just constructed an indecomposable $\theta$-semistable representation that is not $\theta$-stable. Through the equivalence between $\rep(A)$ and $\mathcal{C}$, this $E$ corresponds to an indecomposable $A$-module that is not simple. Now, assume that only one $A_i$ has a non-trivial weight (say $A_1$). Consider $A_2$, which has only trivial weights.
In this case, we can find two extremal vectors $\delta_1, \delta_2$ both supported at a vertex having a given non-trivial weight in $A_1$, at distinct vertices having trivial weights in $A_2$, and at the same vertex having trivial weight in all other $A_i$. We have $\langle \delta_1, \delta_2 \rangle < 0$, which implies that if we take $M_i$ a general representation of dimension vector $\delta_i$, then ${\rm Hom}(M_2, M_1) = {\rm Hom}(M_1, M_2)=0$ while ${\rm Ext}^1(M_1, M_2)$ has dimension $-\langle \delta_1, \delta_2 \rangle > 0$. Similar to the first case, we let
$$0 \to M_2 \to E \to M_1 \to 0$$
be a non-split self-extension. Using that ${\rm Hom}(M_2, M_1) = 0$, we get that $E$ is indecomposable, and we can derive the same conclusion as before on the existence of an indecomposable non-simple $A$-module. Finally, we consider the case where all $A_i$ have trivial weights. In this case, we can find two extremal vectors $\delta_1, \delta_2$ whose support coincide in $m-2$ places. We can check that $\langle \delta_1, \delta_2 \rangle < 0$, and the proof is exactly as in the previous case.
\end{proof}

It follows that when $A$ is infinite dimensional, there exist indecomposable representations that are not simple. However, we will see that an indecomposable $A$-module that is in general position is actually simple. We will prove that a representation of $A$ that is in general position is completely reducible, that is, is isomorphic to a finite direct sum of simple representations.

\medskip

We first need to recall some key notions on $\theta$-stable decompositions. Every finite dimensional $A$-module $M$ has a Jordan--Hölder series. Due to the equivalence between ${\rm rep}(A)$ and $\mathcal{C}$, a Jordan-Hölder series 
$$0 = M_0 \subset M_1 \subset \cdots \subset M_t = M$$ for $M$ corresponds to a chain
$$0 = F(M_0) \subset F(M_1) \subset \cdots \subset F(M_t) = F(M)$$
of subrepresentations of $F(M)$ such that all $F(M_i)/F(M_{i-1})$ are simple in $\mathcal{C}$, or equivalently, are $\theta$-stable in $\mathcal{C}$. Such a composition series for $F(M)$ is closely tied to the notion of $\theta$-stable decomposition. Let $d$ be the dimension vector of $F(M)$. Then $d$ is $\theta$-semistable. Following \cite{DW}, one has a \emph{$\theta$-stable decomposition}
$$d = c_1d_1 + \cdots + c_sd_s$$
of $d$ where the $c_i$ are positive integers such that a general representation $Z$ of dimension vector $d$ admits a filtration of length $t=\sum c_i$ where the subquotients are all $\theta$-stable with exactly $c_i$ such subquotients of dimension vector $d_i$.  The vectors $d_i$ are Schur roots and whenever $d_i$ is imaginary, we have $c_i=1$; see \cite{DW}.

\medskip

Proposition \ref{subrep} guarantees that the stable decomposition of a balanced $\theta$-semistable dimension vector will only use $\theta$-stable dimension vectors that are also balanced as summands. Hence, if $F(M)$ is in general position, one can find the $\theta$-stable composition factors of $F(M)$, and hence the Jordan-Hölder composition factors for $M$. An inductive approach to find the $\theta$-stable decomposition of a dimension vector $d$ is desribed in Remark 6.5 in \cite{DW}.

\begin{proposition}
Let $M$ be a finite dimensional left $A$-module in general position. Then $M$ is completely reducible, that is, $M$ is a direct sum of simple modules.
\end{proposition}

\begin{proof} Let $M$ be as in the statement, so that $F(M)$ is in general position. Let $d$ be the dimension vector of $F(M)$. Let $d_1, d_2, \ldots, d_r$ be the dimension vectors appearing in the $\theta$-stable decomposition of $d$. As previously discussed, we know that the vectors $d_i$ are balanced and $\theta$-stable. It follows from \cite{DW} that the $d_i$ can be ordered in such a way that $\langle d_i, d_j \rangle = 0$ for $i < j$. Using the fact that the Euler form is symmetric for balanced dimension vectors, it follows that $\langle d_i, d_j \rangle = 0$ for $i \ne j$. Note that since $F(M)$ is in general position, all the subquotients $F(M_i)/F(M_{i-1})$ of a Jordan-Hölder filtration are also in general position. If $Z_i$ is a general representation of dimension vector $d_i$, then for $i \ne j$, we have $\Hom(Z_i, Z_j)=0$ since $Z_i, Z_j$ are non-isomorphic $\theta$-stable representations. From $\langle d_i, d_j \rangle =0$, we get $0 = {\rm dim}\Hom(Z_i, Z_j) - {\rm dim}\Ext^1(Z_i, Z_j)$, which means that $\Ext^1(Z_i, Z_j)=0$. The only way for two $\theta$-stable subquotients $F(M_i)/F(M_{i-1})$ and $F(M_j)/F(M_{j-1})$ from above to admit an extension is when they have the same dimension vector, which needs to be imaginary. Indeed, if $Z_{i1}, Z_{i2}$ are two general representations of dimension vector $d_i$, then $\Ext^1(Z_{i1}, Z_{i2})=0$ unless $d_i$ is imaginary. But if $d_i$ is imaginary, we know that its coefficient in the $\theta$-stable decomposition of $d$ is one. Hence, when $F(M)$ is in general position, we get that
$$F(M) \cong \bigoplus_{i=1}^rF(M_i)/F(M_{i-1})$$
is a direct sum of $\theta$-stable representations in $\mathcal{C}$. 
\end{proof}

An $A$-module $M$ is a \emph{Schur} module (also called a \emph{brick}) if ${\rm End}_A(M) = k$. If $M$ is Schur, then $F(M)$ is a Schur representation, hence the dimension vector of $F(M)$ is a Schur root. If $\rep(\Gamma,d)$ contains a brick, then it contains an open set of bricks, since the bricks are the points of $\rep(\Gamma,d)$ with a minimal dimensional ($1$-dimensional here) stabilizer subgroup. Conversely, if $d$ is a balanced Schur root, then a general representation of $\rep(\Gamma, d)$ is a Schur representation in $\mathcal{C}$. Therefore, we get the following.

\begin{corollary}\label{coro:schur}
The balanced Schur roots are exactly the balanced $\theta$-stable dimension vectors.
\end{corollary}

Recall that if $d$ is a dimension vector of $\Gamma$, then there are dimension vectors $d_1, \cdots, d_r$ such that a general representation $M$ of dimension vector $d$ decomposes as a finite direct sum of indecomposable representations $M = M_1 \oplus \cdots \oplus M_r$ such that the dimension vector of $M_i$ is $d_i$. This yields the \emph{canonical decomposition}
$$d = d_1 + \cdots + d_r$$
of $d$ as defined by Kac in \cite{Kac}. We note that all $d_i$ are Schur roots. By the above analysis, we have the following.

\begin{proposition}
The canonical decomposition of a balanced dimension vector agrees with its $\theta$-stable decomposition.
\end{proposition}

\section{Description of all modules and representation type}

Now that we know that left $A$-modules in general position are semisimple, we may ask how to build the other modules. Each finite dimensional module has finitely many simple composition factors. Hence, we can start by fixing a finite set $\mathcal{S} = \{S_1, \ldots, S_r\}$ of finite dimensional simple $A$-modules, and ask for the modules having their simple composition factors in $\mathcal{S}$. We denote by ${\rm wide}(\mathcal{S})$ the abelian extension-closed subcategory of ${\rm rep}(A)$ containing $\mathcal{S}$. It is clear that the simple objects in ${\rm wide}(\mathcal{S})$ are the ones in $\mathcal{S}$.  Now, this category ${\rm wide}(\mathcal{S})$ is a Hom-finite hereditary abelian $\K$-linear category. We let $\delta_i$ denote the dimension vector of $S_i$, and we let $Q(\mathcal{S})$ denote the quiver that has vertices labeled by the $S_i$'s, has $-\langle \delta_i, \delta_j \rangle$ many arrows from $S_i$ to $S_j$ for each pair $i, j$ with $i \ne j$, and has $1 - \langle \delta_i, \delta_i \rangle$ many loops at $S_i$ for each $i$. In other words, the number of arrows from $S_i$ to $S_j$ is just the dimension of $\Ext^1(S_i, S_j)$.

\begin{proposition}
The subcategory ${\rm wide}(\mathcal{S})$ is equivalent to the category of nilpotent representations of $Q(\mathcal{S})$.
\end{proposition}

\begin{proof}
Let $t$ be a positive integer with $t \ge 2$. Consider the full subcategory ${\rm wide}(\mathcal{S})_t$ of ${\rm wide}(\mathcal{S})$ consisting of objects of loewy length at most $t$. This category is closed under subobjects and quotients, hence is abelian. Now, it follows from a result of Gabriel \cite{Gabriel} that a Hom-finite, Ext-finite abelian $\K$-linear category having finitely many simple objects such that the loewy length of each object is bounded by $t$ is equivalent to the module category of a finite dimensional associative $\K$-algebra. Notice that since $t \ge 2$, the dimension of $\Ext^1(S_i, S_j)$ in ${\rm wide}(\mathcal{S})_t$ is the same as the one in ${\rm wide}(\mathcal{S})$. Since ${\rm wide}(\mathcal{S})_t$ has $t$ many simple objects with the dimension of the extensions between the simple objects given by the number of arrows between the corresponding vertices in $Q(\mathcal{S})$, we see that ${\rm wide}(\mathcal{S})_t$ is equivalent to the module category of a quotient $\K Q(\mathcal{S})/I_t$ of the path algebra $\K Q(\mathcal{S})$. Doing so for each $t$, we see that ${\rm wide}(\mathcal{S})$ is a full subcategory of the category of finite dimensional $\K Q(\mathcal{S})$-modules. Now, the only $\K Q(\mathcal{S})$-modules that we get in this subcategory are the ones annihilated by large enough paths, namely the nilpotent ones.
\end{proof}

\begin{example} \label{ExampleTame}
Consider $A = \K^2*\K^2$. In this example, we identify the $A$-modules with the category $\mathcal{C}$. It follows from Theorem \ref{mainThm} that the only $\theta$-stable balanced dimension vectors are $\delta_{ij}$ with $1 \le i,j \le 2$ together with the indivisible special isotropic dimension vector, say $d$.

Consider the one dimensional simple modules $S_1, S_2$ with respective dimension vectors $\delta_{(1,1)}$ and $\delta_{(2,2)}$. Since $\langle \delta_{(1,1)}, \delta_{(2,2)} \rangle = \langle \delta_{(2,2)}, \delta_{(1,1)} \rangle = -1$ and $\langle \delta_{(1,1)}, \delta_{(1,1)} \rangle = \langle \delta_{(2,2)}, \delta_{(2,2)} \rangle=1$, the quiver $Q(\mathcal{S})$ is simply
$$\xymatrix{S_1 \ar@/^/[r] & S_2 \ar@/^/[l]}$$
Now, the category of nilpotent representations of this quiver is formed by the finite dimensional uniserial modules of the form
$$\xymatrixrowsep{5pt}\xymatrix{S_1 \ar@{-}[d] & S_2 \ar@{-}[d]\\ S_2 \ar@{-}[d] & S_1 \ar@{-}[d] \\ S_1 \ar@{-}[d] & S_2 \ar@{-}[d] \\
\vdots & \vdots}$$
Indeed, if $M$ is a nilpotent representation, then it is annihilated by long enough paths. The quotient of $\K Q(\mathcal{S})$ by the ideal of paths of length $r$ is a Nakayama algebra, and it is well known that all modules over such an algebra are uniserial as above. A similar analysis holds if we start with the one dimensional simple modules $S_1, S_2$ with respective dimension vectors $\delta_{(1,2)}$ and $\delta_{(2,1)}$.

If $S,T$ are simple $A$-modules of distinct dimension vectors, then the only way they can have an extension is when their respective dimension vectors are $\delta_{(1,1)}, \delta_{(2,2)}$; or $\delta_{(1,2)}, \delta_{(2,1)}$. If they are of the same dimension and non-isomorphic, then they do not admit extensions. Also, if $S$ is of dimension vector $\delta_{(i,j)}$, then it has no self-extensions. For each simple $A$-module $S$ of dimension vector $d$, the category ${\rm wide}(S)$ is equivalent to the category of nilpotent representations of the one-loop quiver, or equivalently, to the graded finite dimensional modules over $\K[x]$. Hence, all indecomposable representations in $\mathcal{C}$ can be classified as follows. They are either as above for $S_1, S_2$ either of dimension vectors $\delta_{(1,1)}, \delta_{(2,2)}$ or of dimension vectors $\delta_{(1,2)}, \delta_{(2,1)}$; or of the form
$$\xymatrixrowsep{5pt}\xymatrix{S_\lambda \ar@{-}[d] \\ S_\lambda \ar@{-}[d] \\ S_\lambda \ar@{-}[d] \\ \vdots}$$
for $S_\lambda$ one of the simple $A$-module of dimension vector $d$ indexed by a $1$-parameter family. The readers familiar with representation theory of quivers might recognize that we have obtained indecomposable regular representations of the subspace quiver of type $\widetilde{\mathbb{D}}_4$. In fact, $\mathcal{C} \cong \rep(A)$ is nothing but the category of regular representations of the subspace quiver of type $\widetilde{\mathbb{D}}_4$.\qed
\end{example}

We can use this to study the representation type of $\rep(A)$ as follows. Recall Drozd's dichotomy theorem in the setting of finite dimensional associative $\K$-algebras. It asserts that such an algebra is either tame, or wild, but not both. For some classes of finite dimensional algebras, and in particular for hereditary algebras, wild is equivalent to a stronger condition, namely that of strictly wild. A finite dimensional associative $\K$-algebra $A$ is \emph{strictly wild} if there exists a fully faithful exact embedding $\varphi: {\rm mod}\K K_3 \to {\rm mod}A$, where $K_3$ is the $3$-Kronecker quiver. Such an embedding preserves indecomposability and isomorphism classes. Note that a strictly wild algebra is wild as in Drozd's dichotomy theorem. If $A$ is strictly wild, then for any finite dimensional associative $\K$-algebra $B$, there is a fully faithful exact embedding ${\rm mod}B \to {\rm mod}A$. Hence, for such an algebra $A$, there is no hope in classifying the indecomposable finite dimensional modules, since such a classification would yield a classification of all indecomposable finite dimensional modules of all finite dimensional associative $\K$-algebras! 

Based on this, we say that our (generally infinite dimensional) algebra $A$ is \emph{strictly wild} if there exists a fully faithful exact embedding $\varphi: {\rm mod}\K K_3 \to {\rm mod}A := \rep A$.

In the example above, we see that for any balanced dimension vector $d$, all but finitely many indecomposable modules are part of a $1$-parameter family of indecomposable modules. Based on the similar notion for finite dimensional algebras, we say that our algebra $A$ is of \emph{tame} representation type if for any dimension (vector) $d$, all but finitely many indecomposable modules are in finitely many $1$-parameter families of indecomposable modules. Here, we do not go in the details of what is meant by a one-parameter family, because of the following theorem, which says that the only example of tame representation type is as in the previous example.

\begin{theorem} Let $A$ be a free product of finitely many non-trivial semisimple associative $\K$-algebras. 
\begin{enumerate}
    \item If $A$ is not isomorphic to $\K^2 * \K^2$, then $\rep(A)$ is of strictly wild representation type and has unbounded simples.
    \item If $A = \K^2 * \K^2$, then $\rep(A)$ is of tame representation type. The simple modules are $1$ or $2$ dimensional, and indecomposable modules can be classified as in Example \ref{ExampleTame}.
\end{enumerate}
\end{theorem}

\begin{proof}
If there are at least two non-trivial weights in distinct clusters, then we have a balanced imaginary Schur root $d$ of $\Gamma$. In all other cases, except for $\K^2 * \K^2$, we can built a balanced $\theta$-stable dimension vector $d$ that is not on an extremal ray of $\mathbb{Q}_+\Sigma(\Gamma, \theta)$. This dimension vector is necessarily imaginary.

Any integral multiple of $d$ is again a balanced imaginary Schur root. For $m$ large enough, $\langle md, md \rangle \le -3$. If we take $2$ non-isomorphic $\theta$-stable representations $S_1, S_2$ of $\mathcal{C}$ of dimension vector $md$, then the quiver $Q(\{S_1, S_2\})$ has at least $3$ arrows from $S_1$ to $S_2$ and at least $3$ arrows from $S_2$ to $S_1$ (and at least $2$ loops at each vertex). Hence ${\rm mod}\K K_3$ is a full subcategory of ${\rm wide}(\{S_1, S_2\})$, hence can be seen as a full subcategory of $\rep(A)$. 

The case of $\K^2 * \K^2$ was treated in Example \ref{ExampleTame}.
\end{proof}

\section{Number of parameters}

An interesting consequence of the previous analysis is that to each $\theta$-stable dimension vector $d$ of $\Gamma$, following King \cite{Ki}, we can associate a projective variety $\mathcal{M}^s(\Gamma,d)$ which parametrizes the isomorphism classes of the $\theta$-stable representations of $\Gamma$ of dimension vector $d$. When non-empty, this variety has dimension $1 - \langle d, d \rangle$. When $d$ is balanced, since $\mathcal{C}_d$ is an open set of the affine space $\rep(\Gamma,d)$, this implies that $1 - \langle d, d \rangle$ equals the number of parameters for the isomorphism classes of simple representations in $\rep(A,d)$.  Thus we have the following.

\begin{proposition}
If $d$ is a balanced dimension vector and is $\theta$-stable, then the number of parameters for the isomorphism classes of simple representations in $\rep(A,d)$ is $1 - \langle d, d \rangle$.
\end{proposition}

When all weights are trivial, we can easily deduce the number of parameters for the simple representations of $\rep(A,n)$. Let $n \ge 1$ and $r_1, \ldots, r_m$ be given, and assume that $w_{ij}=1$ for all $i,j$.  We define $d_{n, \rm min}$ to be a balanced dimension vector such that $d_0 = n$ and for each $i$, $|d_{ij}-d_{ik}| \le 1$ for all $j,k$. This dimension vector is unique up to ordering of the entries of $d$ within each cluster. 

\begin{corollary}
Assume that all $A_i$ have trivial weights. For $n \ge 1$ such that $\rep(A,n)$ has simple representations, the number of parameters of simple representations in $\rep(A,n)$ is $$1 - \langle d_{n,\rm min}, d_{n,\rm min} \rangle=1 + (m-1)n^2 - \sum_{i,j}d_{ij}^2$$
where $d_{ij}$ is the dimension at $v_{ij}$ of $d_{n,\rm min}$.
\end{corollary}

\begin{proof}
In order to minimize $\langle d, d\rangle$ for $d$ balanced such that $d_0=n$, one just needs to take $d$ with $d_0 = n$ such that for each cluster $i$, the sum $\sum_{j=1}^{r_i}d_{ij}^2$ is maximal. This is attained for $d = d_{n,\rm min}$. Now, the claimed equality
follows from direct computations.
\end{proof}

\begin{example}
Let $m=2$, $r_1=2$ and $r_2 = 3$ with all weights trivial. Here, $A$ is the free product of $\K^2$ with $\K^3$. Consider $n=8$. Here, $d_{8, \rm min}= (8, 4, 4, 3, 3, 2)$. The value of $1 + (m-1)n^2 - \sum_{i,j}d_{ij}^2$ is $11$. Hence, the $8$-dimensional simple $A$-modules, up to isomorphism, can be parametrized using at most $11$ parameters.
\end{example}

\section{Free products of finite groups}

In this section, we let $G_1, \ldots, G_m$ be finite groups and we assume that the characteristic of $\K$ does not divide the order $|G_i|$ of $G_i$ for all $i$. In this case, it follows from Maschke's theorem that all $A_i:= \K G_i$ are semisimple. Therefore, our previous analysis applies for
$$A:= (\K G_1) * \cdots * (\K G_m) \cong \K G$$
where $G$ is the free product $G_1 * \cdots * G_m$. Of particular interest among free products of finite groups are the infinite dihedral group $\mathbb{Z}_2 * \mathbb{Z}_2$ and the modular group $\mathbb{Z}_2 * \mathbb{Z}_3\cong {\rm PSL}_2(\mathbb{Z})$. These two groups are free products of cyclic groups. We will look at cyclic groups in more details in the next subsection.

Now, let $M$ be an $n$-dimensional left $A$-module given by the vector space $V$. While it may not be straightforward to compute $F(M)$, if one is not explicitly given the matrix form of each $\K G_i$, Serre in \cite[Section 2.6]{Se} gives a way to compute the canonical decomposition of a representation $M: G_i \to  {\rm GL}_\K(V)$ as
$$V = V_{i1} \oplus \cdots \oplus V_{ir_i}$$
which is such that $V_{ij}$ is the maximal direct summand of $V$ isomorphic to a direct sum of copies of the $j$-th simple $\K G_i$-module. These modules $V_{ij}$ are unique and called \emph{isotypic} components of the $\K G_i$-representation $M$. When one has an internal decomposition $\K G_i = A_{i1} \times \cdots \times A_{ir_i}$ in matrix factors, the subspace $V_{ij}$ is nothing but $E_{ij}V$, which is the projection of $V$ onto the $j$-th factor $A_{ij}$.  Finding the summands $V_{ij}$ only requires knowing the (irreducible) characters of $G_i$.

\begin{example}
Consider $m=2$ with $G_1$ the symmetric group of order $6$ and $G_2$ the cyclic group of order $2$. We assume that the characteristic of $\K$ is not in $\{2,3\}$. We have that $\K G_1 \cong \K \times \K \times {\rm Mat}_{2\times 2}(\K)$ and $\K G_2 \cong \K \times \K$. Hence, $r_1 = 3$ with $w_{11} = w_{12} = 1$ and $w_{13}=2$ and $r_2 = 2$ with $w_{21} = w_{22}=1$. Let $M$ be an $n$-dimensional left $A$-module given by the vector space $V$. For each $i$, we are given a group representation
$$\rho_i: G_i \to {\rm GL}_\K(V)$$
representing $M$ as an $\K G_i$-module.

Let $S_{11}, S_{12}, S_{13}$ be the non-isomorphic simple $\K G_1$-modules where $S_{13}$ is $2$-dimensional, and let $S_{21}, S_{22}$ be the non-isomorphic simple $\K G_2$-modules. Suppose for instance that $M$ is $6$-dimensional and is such that, as a $\K G_1$-module, we have
$$M \cong S_{11}^3 \oplus S_{12} \oplus S_{13}$$
while as a $\K G_2$-module, we have
$$M \cong S_{21}^2 \oplus S_{22}^4.$$
Then the quiver $\Gamma$ with the dimension vector of $F(M)$ are represented as follows:
$$
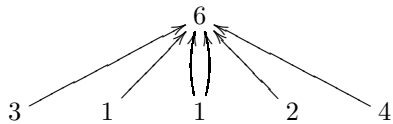
\xymatrix{& & 6 & & \\ 3 \ar[urr] & 1 \ar[ur] & 1 \ar@/^0.3pc/[u] \ar@/_0.3pc/[u] & 2 \ar[lu] & 4 \ar[llu]}$$
Note that $F(M)_{13}$ is $2$-dimensional, and hence, the corresponding dimension at vertex $v_{13}$ is $2/w_{13}=1$. Since ${\rm max}(3,1,1/2) = 3$ and ${\rm max}(2,4)=4$ with the sum $3+4 > (m-1)n = 6$, we see that $M$ cannot be simple. On the other hand, let  $N$ be also $6$-dimensional but such that, as a $\K G_1$-module, we have
$$N \cong S_{11}^2 \oplus S_{12}^2 \oplus S_{13}$$
while as a $\K G_2$-module, we have
$$N \cong S_{21}^2 \oplus S_{22}^4.$$
Now, we have that ${\rm max}(2,2,1/2) = 2$ and ${\rm max}(2,4)=4$ with the sum $2+4 \le  6$. We cannot tell whether $N$ is simple or not. However, if $N$ is in general position satisfying the above two decompositions, then it will be simple.  
\end{example}

For the remaining of this section, we focus our attention to free products of finite cyclic groups.
The group algebra of a finite cyclic group enjoys the additional property of being basic, thus has only trivial weights. There is a convenient way to present the corresponding algebra $A$. For each $i$, let $a_i$ be a generator of $G_i$ of order $r_i$ - here $r_i = |G_i|$ is the number of simple factors in $\K G_i$. We let $Q$ be the quiver with a single vertex $1$ and with $m$ loops $a_1, \ldots, a_m$. The algebra $A$ is nothing but the path algebra $\K Q$ of $Q$ modulo the ideal generated by the relations $a_i^{r_i}-1$ for all $i=1, \ldots, m$. Hence, an $n$-dimensional left $A$-module $M$ is given by an $n$-dimensional vector space $V$ together with $m$ endomorphisms $M(a_i):=M_i$ satisfying $M_i^{r_i} = I_n$ for all $i$, where $I_n$ is the identity $n \times n$ matrix. 

In this case, there is an alternative description of the functor $F$. Indeed, for each $i$, we let $\rho_{i1}, \ldots, \rho_{ir_i}$ denote the $r_i$-th roots of unity. These roots are in one-to-one correspondence with the non-isomorphic (one-dimensional) simple $\K G_i$-modules - say root $\rho_{ij}$ corresponds to simple module $S_{ij}$. Now, for each $i$, $V$ can be decomposed as a direct sum of its eigenspaces with respect to the endomorphism $M_i$:
$$V = F_{i1} \oplus \cdots \oplus F_{ir_i},$$
where we allow zero dimensional eigenspaces. This decomposition agrees with Serre's canonical decomposition above. In other words, the isotypic components here are just the eigenspaces. Each of the $F_{ij}$ is a $\K G_i$-modules that is a direct sum of ${\rm dim}F_{ij}$ copies of $S_{ij}$. Therefore, we see that the functor $F$ is such that $F(M)_{ij}=F_{ij}$ and $F(M)(\alpha_{ij}^1)$ is simply the inclusion of $F_{ij}$ into $V$. Our main theorem can be interpreted as follows. 

In this theorem, for an $n$-dimensional representation $M$ of $G$, we put $d(M)_{ij}$ to be the dimension of the $\rho_{ij}$-eigenspace of the endomorphism $M(a_i)=M_i$. Equivalently, these numbers define a dimension vector $d = d(M)$ of $\Gamma$ that is nothing but the dimension vector of $F(M)$.

\begin{theorem} \label{MainTheoCyclic}Let $G$ be a free product of the finite cyclic groups $G_1, \ldots, G_m$ with $G_i$ of order $r_i$. Assume that for every $i$, the  characteristic of $\K$ does not divide $r_i$. Let $n$ be a positive integer and consider $M$ and $d$ as in the previous paragraph.

\begin{enumerate}
\item The dimension vector $d$ is $\theta$-stable if and only if either 
\begin{enumerate}
    \item $n=1$ or
    \item $d$ is special isotropic indivisible (so $n=2$ and precisely $4$ entries in $d$ are equal to $1$), or
    \item $d$ is not special isotropic and 
$$\sum_{i=1}^m{\rm max}(d_{i1}, \ldots, d_{i r_i}) \le (m-1)n.$$
\end{enumerate}
    \item If $M$ is simple, then $d$ is $\theta$-stable, hence satisfies (1).
    \item 
   If $d$ is $\theta$-stable, then the space of $n$-dimensional representations $M$ with $F(M)$ having dimension vector $d$ has a Zariski-dense and open subset of simple representations.
\end{enumerate}
\end{theorem}

\begin{proof}
We apply Theorem \ref{mainThm} by noticing that all $\delta_u$ are real and that $d$ is equal to some $\delta_u$ if and only if $n=1$.
\end{proof}


\begin{thebibliography}{99999}

\bibitem{AL} J. Adriaenssens and L. Le Bruyn, {\em Local quivers and stable representations},
Comm. Algebra {\bf 31} (2003), no. 4, 1777--1797.

\bibitem{CJ} C. Chindris and J. Ismaeel, {\em A quiver invariant theoretic approach to Radial Isotropy and Paulsen's Problem for matrix frames}, preprint, arXiv:2104.11310.


\bibitem{DPWX} I. Dimitrov, C. Paquette, D. Wehlau and T. Xu, {\em Subregular J-rings of Coxeter systems via quiver path algebras}, preprint, arXiv:2101.06851.

\bibitem{DW} H. Derksen and J. Weyman, {\em The combinatorics of quiver representations}, Ann. Inst. Fourier (Grenoble) {\bf 61} (2011), no. 3, 1061--1131.

\bibitem{Gabriel} P. Gabriel, {\em Indecomposable representations. II}, in Symposia Mathematica, Vol. XI (Convegno
di Algebra Commutativa, INDAM, Rome, 1971), 81–104, Academic Press, London, 1973.

\bibitem{Kac} V.G. Kac, {\em Infinite root systems, representations of graphs and invariant theory. II.} J. Algebra {\bf 78} (1982), no. 1, 141--162. 

\bibitem{Ki} A.D. King, {\em Moduli of representations of finite dimensional algebras},
Quart. J. Math. Oxford Ser. {\bf 45 (2)} (1994), no. 180, 515--530.

\bibitem{Se} J.P. Serre, {\em Linear representations of finite groups},
(English translation) Graduate Texts in Mathematics, Vol. 42. Springer-Verlag, New York-Heidelberg, 1977. x+170 pp.

\bibitem{Slet} A.B. Sletsjow, {\em Finite dimensional representations of the projective modular group}, 2006, preprint.
\end{thebibliography}
\end{document}